\documentclass[reqno,centertags,11pt,a4paper]{amsart}
\usepackage{amssymb}
\usepackage{mathrsfs}
\usepackage{amsmath, amsfonts,vmargin, enumerate}
\usepackage{color}
\usepackage{amsthm}
\usepackage{latexsym,bm}
\usepackage{euscript}
\usepackage{color}
\usepackage[colorlinks=true,citecolor=blue,linkcolor=blue]{hyperref}

\begin{document}
	\newtheorem{theorem}{Theorem}
	\newtheorem{proposition}[theorem]{Proposition}
	\newtheorem{conjecture}[theorem]{Conjecture}
	\newtheorem{corollary}[theorem]{Corollary}
	\newtheorem{lemma}[theorem]{Lemma}
	\newtheorem{sublemma}[theorem]{Sublemma}
	\newtheorem{observation}[theorem]{Observation}
	\newtheorem{remark}[theorem]{Remark}
	\newtheorem{definition}[theorem]{Definition}
	\theoremstyle{definition}
	\newtheorem{notation}[theorem]{Notation}
	\newtheorem{question}[theorem]{Question}
	\newtheorem{example}[theorem]{Example}
	\newtheorem{problem}[theorem]{Problem}
	\newtheorem{exercise}[theorem]{Exercise}
	\numberwithin{theorem}{section} 
	\numberwithin{equation}{section}
	
	\title[Maximal Riesz transform in terms of Riesz transform]{Maximal Riesz transform in terms of Riesz transform on quantum tori and Euclidean space}

	\author{Xudong Lai}
	\address{Institute for Advanced Study in Mathematics, Harbin Institute of Technology, 150001 Harbin, China; Zhengzhou Research Institute\\
		Harbin Institute of Technology\\
		Zhengzhou
		450000\\
		China}
	\email{xudonglai@hit.edu.cn}
	\author{Xiao Xiong}
	\address{Institute for Advanced Study in Mathematics, Harbin Institute of Technology, 150001 Harbin, China}
	\email{xxiong@hit.edu.cn}
	
	\author{Yue Zhang}
	\address{Institute for Advanced Study in Mathematics, Harbin Institute of Technology, 150001 Harbin, China}
	\email{21b912030@stu.hit.edu.cn}
	
	%
	
	\date{}

	\markboth{ }
	{Maximal Riesz transform in terms of  Riesz transform}
	\begin{abstract}
		For $1<p<\infty$, we establish the $L_{p}$ boundedness of  the maximal Riesz transforms in terms of the Riesz transforms on quantum tori  $L_{p}(\mathbb{T}^{d}_{\theta})$,  and quantum Euclidean space  $L_{p}(\mathbb{R}^{d}_{\theta})$. In particular, the norm constants in both cases are independent of the dimension  $d$ when $2\leq p<\infty$. 
	\end{abstract}
	\maketitle
	\section{Introduction}

	The Riesz transforms $R_{j}$, $1\leq j\leq d$, play  a crucial role in classical harmonic analysis and have been explored from various perspectives  (see e.g. \cite{MR4778175, MR4585167,liu2023p,MR2280788,MR4631033}). The aim of this paper is to continue the work of  Mateu and Verdera \cite{MR2280788}, Liu, Melentijevi\'{c}  and Zhu \cite{liu2023p}  on  the $L_{p}$ boundedness   of the maximal Riesz transforms  in terms of the Riesz transforms, however in the noncommutative setting. Let us first recall these results.   For a Schwartz function  $f$,    its  $j$-th Riesz transform   is defined by 
	\begin{align*}
		R_{j}f(x) =\lim_{\varepsilon\rightarrow 0}R_{j}^{\varepsilon}f(x),
	\end{align*}
	where 
	\begin{align}\label{24312.1} 
		R_{j}^{\varepsilon}f(x) =\frac{\Gamma(\frac{d+1}{2})}{\pi^{\frac{d+1}{2}}}\int_{|x-y|>\varepsilon}f(y)\frac{x_{j}-y_{j}}{\,|x-y|^{d+1}}dy 
	\end{align} 
	is the associated truncated Riesz transform.  Let 
	\begin{align*} 
		R_{j}^{*}f(x)=\sup_{\varepsilon>0}|R_{j}^{\varepsilon}f(x)|
	\end{align*}
	be the maximal Riesz transform. It is clear that the pointwise inequality $|R_{j}f(x)|\leq R^{*}_{j}f(x)$ holds for  Schwartz functions.  As for the reverse inequality,  Mateu and Verdera \cite{MR2280788}  showed that there exists a constant $C_{d,p}$  such that 
	\begin{align}\label{24520.2}
		\|R_{j}^{*}f\|_{L_{p}(\mathbb{R}^d)}\leq C_{d,p}\|R_{j} f\|_{L_{p}(\mathbb{R}^d)},~~~1<p<\infty.
	\end{align}
	
	A natural problem is whether  the constant $C_{d,p}$ in  (\ref{24520.2}) is independent of the dimension $d$.  Kucharski and Wr\'{o}bel \cite{MR4585167} gave a positive answer in the case of $p=2$.
	They used the Littlewood-Paley decomposition technique to obtain a universal $L_{2}$ estimate $C_{d,2}\leq 2\cdot10^{8}$. Further, for $2\leq p\leq\infty$, Liu et al. \cite{liu2023p}   improved the estimate of $C_{d,p}$ as
	\begin{align*}
		\|R_{j}^{*}f\|_{L_{p}(\mathbb{R}^d)}\leq \bigg(2+\frac{\sqrt{2}}{2}\bigg)^{\frac{2}{p}}\|R_{j} f\|_{L_{p}(\mathbb{R}^d)}.
	\end{align*}
	Up to now, the dimension-free estimate of the maximal Riesz transforms in terms of the Riesz transforms for $1<p<2$ is still an open problem.
	
	In recent years, the noncommutative harmonic analysis has developed rapidly  based on  operator algebra,  noncommutative geometry, and quantum mechanics
	(see e.g. \cite{MR3079331,MR2574890,hong2022noncommutative,1,MR4623347,MR4585152,MR1916654,MR4381188,MR4746875,MR4029831,MR4156216,MR2327840}).    
	Due to the noncommutativity, the study of the noncommutative theory requires  novel ideas  and techniques. For instance,  let $\mathcal{M}$ be a von Neumann algebra and $B$ the unit ball in $\mathbb{R}^d$. Define 
	\begin{align}\label{241227.2}
		M_{t}f(x)=\frac{1}{|B|}\int_{B}f(x-ty)dy  
	\end{align}
	acting on  some operator-valued function $f:\mathbb{R}^d\rightarrow \mathcal{ {M}}$, which belongs to the noncommutative $L_{p}$ space.   The Hardy-Littlewood maximal operator $Mf=\sup_{t>0}M_{t}f$ for a sequence of operators $(M_{t}f)_{t>0}$ seems to be unavailable since we can not directly compare two operators in $\mathcal{M}$. Thus it is much more subtle to explain its  maximal $L_{p}$ inequality in the noncommutative  setting. This obstacle was not overcome until the introduction of Pisier’s vector-valued noncommutative   spaces $L_{p}(L_{\infty}(\mathbb{R}^d)\overline{\otimes}\mathcal{M};l_{\infty})$ (see \cite{MR1648908}), which came nearly three decades after the first appearance of two  non-trivial maximal weak $(1,1)$ inequalities: one by   Cuculescu \cite{MR295398} for noncommutative martingales in 1971,  and another by Yeadon \cite{MR487482}  for ergodic averages in 1977. Following Pisier’s idea, the  $L_{p}$ norm of the   Hardy-Littlewood maximal operator is understood as $L_{p}(L_{\infty}(\mathbb{R}^d)\overline{\otimes}\mathcal{M};l_{\infty})$ norm of  the sequence $(M_{t}f)_{t>0}$ in the noncommutative case.  On this basis,  Mei \cite{MR2327840} showed  the one-dimensional Hardy-Littlewood maximal operator inequality,  relying  on the noncommutative Doob maximal inequality (see \cite{MR1916654}) and two increasing   filtrations  of dyadic $\sigma$-algebras on $\mathbb{R}$. Junge and Xu \cite{MR2276775} established the $L_{p}$ maximal inequality for ergodic average by    employing a noncommutative analogue of the Marcinkiewicz interpolation theorem.

	In the present paper,  we are  interested in  quantum tori and Euclidean space,  and aim to establish  the boundedness of the maximal Riesz transforms in terms of the Riesz transforms on them. 
	
	Quantum tori, as the fundamental examples in the noncommutative geometry,  have been extensively studied  (see e.g. \cite{MR1247054,MR1047281,MR2239597}). The analysis of quantum tori began in \cite{MR1154104,MR1402766}, and the first systematic study of harmonic analysis on quantum tori was presented in \cite{MR3079331}. For recent developments in this topic,  we refer to  \cite{MR3778347,MR4381188,MR4029831,MR3490779,MR3459017,MR3778570} and the references therein.

	To  present our first result, we   introduce  the  definition of quantum tori.  Fix $d\geq2$,  and $\theta=(\theta_{k,l})_{1\leq k,l\leq d}$ is a real   antisymmetric  $d\times d$ matrix. The  $d$-dimensional noncommutative torus $\mathcal{A}_{\theta }$ is the   universal $C^{*}$-algebra generated by $d$ unitary operators $V_{1},\cdots,V_{d}$  satisfying 
	\begin{align*}
		V_{k}V_{l}=e^{2\pi\mathrm{i}\theta_{k,l}}V_{l}V_{k}, ~~~1\leq k,l\leq d.
	\end{align*}
	The algebra  $\mathcal{A}_{\theta }$ admits a faithful tracial state $\tau_{\theta}$.  The $d$-dimensional quantum torus, denoted by  $\mathbb{T}_{\theta }^{d}$, is the weak $*$-closure of $\mathcal{A}_{\theta }$ in the GNS representation of $\tau_{\theta}$. The state  $\tau_{\theta}$    is extended to a normal  faithful tracial state on $\mathbb{T}_{\theta }^{d}$. Note that if $\theta=0$,  $\mathcal{A}_{\theta }=C(\mathbb{T}^{d})$ and $\mathbb{T}_{\theta }^{d}=L_{\infty}(\mathbb{T}^d)$, where $\mathbb{T}^d$ denotes the classical 
	$d$-dimensional tours. Thus   quantum torus $\mathbb{T}_{\theta }^{d}$ could be viewed as a deformation of   $\mathbb{T}^d$.

	Let  $L_{p}(\mathbb{T}^d_{\theta })$ be the  noncommutative $L_{p}$ space associated with pairs $(\mathbb{T}_{\theta}^{d},{\tau}_{\theta })$. For any $\mathbf{x}\in L_{1}(\mathbb{T}^d_{\theta })$ and $n\in\mathbb{Z}^d$, 
	\begin{align*}
		\widehat{\mathbf{x}}(n)={\tau}_{\theta }((\mathbf{V}^n)^{*}\mathbf{x}) 
	\end{align*}
	is called the Fourier coefficient of $\mathbf{x}$, where $\mathbf{V}^{n}=V_{1}^{n_{1}}\cdots V_{d}^{n_{d}}$. The standard Hilbert space arguments
	show that any $\mathbf{x}\in L_{2}(\mathbb{T}^d_{\theta })$ can be written as an $L_2$-convergent series:
	\begin{align*}
		\mathbf{x}=\sum_{n\in\mathbb{Z}^d}\widehat{\mathbf{x}}(n)\mathbf{V}^n,
	\end{align*}
	with the Plancherel formula
	\begin{align}\label{241210.1}
		\|x\|_{ L_{2}(\mathbb{T}^d_{\theta })}=\bigg(\sum_{n\in \mathbb{Z}^d}|\widehat{\mathbf{x}}(n)|^2\bigg)^{\frac{1}{2}}.
	\end{align}
	
	Now, we define the  noncommutative   Riesz transforms  on quantum tori. For this purpose, we   recall the   classical   form  on $\mathbb{T}^d$:
	\begin{align}\label{241117.8}
		R_{j}f(x)=\sum_{n\in \mathbb{Z}^d}(-\mathrm{i})\frac{n_{j}}{|n|}\widehat{f}(n)e^{2\pi\mathrm{i} \langle n,x\rangle},~~~f\in L_{1}(\mathbb{T}^d),
	\end{align} 
	where  $\widehat{f}(n)=\int_{\mathbb{T}^d}f(x)e^{-2\pi\mathrm{i}\langle   n,x\rangle }dx$. Note that on quantum tori,   $\mathbf{V}^n$  plays the role of $e^{2\pi \mathrm{i}   \langle n,\cdot\, \rangle}$.   Thus given   $\mathbf{x}\in L_{2}(\mathbb{T}^d_{\theta })$,    the $j$-th noncommutative  Riesz transform   of $ \mathbf{x}$ is   defined by
	\begin{align*}
		R_{j}\mathbf{x}=\sum_{n\in \mathbb{Z}^d}(-\mathrm{i})\frac{n_{j}}{|n|}\widehat{\mathbf{x}}(n)\mathbf{V}^n.
	\end{align*}
	The Plancherel identity (\ref{241210.1})  ensures that the noncommutative  Riesz transforms are   bounded linear operators on $L_{2}(\mathbb{T}^d_{\theta })$ and the above series   converges in the $L_{2}$ sense.  We follow the same approach to define  the $j$-th noncommutative truncated Riesz transform, which reduces to providing the associated multiplier.  Let   $K_{j}^{\varepsilon}(x)=\frac{\Gamma(\frac{d+1}{2})}{\pi^{\frac{d+1}{2}}}\chi_{|x|>\varepsilon}(x)\frac{x_{j}}{|x|^{d+1}}$ denote   the  convolution kernel of  (\ref{24312.1}).   It was shown in \cite[Lemma 3.1]{MR4585167} that
	\begin{align}\label{24329.4}
		\widehat{K_{j}^{\varepsilon}}(\xi)=(-\mathrm{i})\frac{\xi_{j}}{|\xi|}m(\varepsilon|\xi|), ~~~\xi\in\mathbb{R}^d,~\xi\neq0,
	\end{align}
	with
	\begin{align}\label{240612.1}
		m(x)=\frac{1}{\pi}\int^{\infty}_{2\pi x}\int_{-1}^{1}e^{\mathrm{i}rs}(1-s^2)^{\frac{d-1}{2}}dsdr,~~x>0.
	\end{align}
	Thus, for any $\varepsilon>0$, we define the $j$-th  noncommutative truncated Riesz transform of $\mathbf{x}\in L_{2}(\mathbb{T}^d_{\theta })$   as
	\begin{align*}
		R_{j}^{\varepsilon}\mathbf{x}=\sum_{n\in \mathbb{Z}^d}(-\mathrm{i})\frac{n_{j}}{|n|}m(\varepsilon|n|)\widehat{\mathbf{x}}(n)\mathbf{V}^n.
	\end{align*}
	
	As for the maximal Riesz  transform,   its $L_{p}(\mathbb{T}^d_{\theta })$ norm   should be understood as the  $L_{p}(\mathbb{T}^d_{\theta };l_{\infty}(\mathbb{R}_{+}))$ norm of the sequence $(R_{j}^{\varepsilon}\mathbf{x})_{\varepsilon}$ in the noncommutative case,  denoted by $\|\sup_{\varepsilon}R_{j}^{\varepsilon}\mathbf{x}\|_{L_{p}(\mathbb{T}^d_{\theta })}$. It should be noted
	that $\sup_{\varepsilon}R_{j}^{\varepsilon}\mathbf{x}$ is just a notation as it does not make sense; we  refer  the reader  to Section \ref{24521.1} for the concrete definition of $L_{p}(\mathcal{M}; l_{\infty})$.    
	
	We are ready to state our results  on quantum tori. 
	\begin{theorem}\label{241117.3}
		Let $1<p<\infty$ and $\mathbf{x}\in L_{p}(\mathbb{T}^d_{\theta})$, then there exists a constant $C_{d,p}$ such that 
		\begin{align}\label{241112.3}
			\|\sup_{\varepsilon>0}R_{j}^{\varepsilon}\mathbf{x}\|_{L_{p}(\mathbb{T}^d_{\theta})}\leq C_{d,p}\|R_{j}\mathbf{x}\|_{L_{p}(\mathbb{T}^d_{\theta})}.
		\end{align}
		Moreover, if $2\leq p<\infty$,   the  constant $C_{d,p}$ in  $(\ref{241112.3})$ is independent of the dimension,  
		\begin{align}\label{241112.4}
			C_{d,p}\leq (96+2\sqrt{2})^{\frac{2}{p}}.
		\end{align}
	\end{theorem} 
	\vskip 0.24cm
	
	The second purpose of this paper is to obtain the boundedness of the  maximal Riesz transforms in terms of the Riesz transforms on the quantum Euclidean space.  This space originated from the study of quantum mechanics in phase space in the 1940s (see \cite{MR18562,MR29330}). As a significant  object in the noncommutative geometry,   quantum Euclidean space  serves as a model example in the literature of   mathematical physics (see \cite{MR2053945}). It has also been explored as an interesting noncommutative setting for   harmonic analysis, including  Fourier restriction  phenomena (see \cite{MR4623347}) and  quantum  differentiability (see  \cite{MR4156216}). Compared with the quantum tori, the study of  the  quantum Euclidean space is more complicated. Since the noncommutative space  $L_{p}(\mathbb{R}_{\theta}^{d})$ can not be embedded isometrically into $L_{p}(\mathbb{R}^d;L_{p}(\mathbb{R}_{\theta}^{d}))$ via a natural  isomorphic mapping,   the results in the semi-commutative case can not  be directly applied to the quantum Euclidean space. However, in  some research fields,  the noncommutative analogue of Calder\'{o}n’s transference principle (see \cite[Theorem 3.1]{MR4202493})  serves as a  useful tool. Recall that  this classical principle  appeared in \cite{MR227354}  and   was  extended   to the noncommutative setting  with the aim of establishing the ergodic theorem for group actions on von Neumann algebras in \cite{MR4202493}.   Hong et al. \cite{1} applied the noncommutative Calder\'{o}n’s transference principle to develop  some ergodic theory on quantum Euclidean space,  building  on  the semi-commutative Calder\'{o}n-Zygmund theory  (see \cite[Theorem 1.3]{MR4585152}).

	To give the concrete definition of the Riesz transforms on the quantum Euclidean space  and state our results,    we  need to provide  a brief introduction to the   quantum Euclidean space. Suppose $d\geq 2$, and $\theta$ is a fixed antisymmetric $d\times d$ matrix. The quantum Euclidean space, denoted by $\mathbb{R}_{\theta}^{d}$, 
	is a von Neumann algebra of $\mathcal{B}(L_{2}(\mathbb{R}^d))$ generated by   unitary operators $(U(t))_{t\in\mathbb{R}^d}$, which  act  on   $L_{2}(\mathbb{R}^d)$ as follows:  
	\begin{align}\label{241212.1}
		(U(t)f)(r)=e^{-\frac{\mathrm{i}}{2} \langle t,\theta r\rangle}f(r-t),~~f\in L_{2}(\mathbb{R}^d),~~r\in\mathbb{R}^d.
	\end{align}
	One can verify that (\ref{241212.1}) satisfies Weyl relation: $U(t)U(s)=e^{\frac{\mathrm{i}}{2} \langle t,\theta s \rangle }U(t+s).$ 
	
	Let  $f\in L_{1}(\mathbb{R}^d)$,  we construct the operator $ U(f)\in \mathbb{R}_{\theta}^{d}$ in the following manner:
	\begin{align*}
		U(f)g=\int_{\mathbb{R}^d}f(t)U(t)gdt,~~g\in L_{2}(\mathbb{R}^d).
	\end{align*}
	The above $L_{2}(\mathbb{R}^d)$-valued  integral is absolutely convergent in the Bochner sense.  	The following image of the Schwartz 
	function $\mathcal{S}(\mathbb{R}^{d})$ under the map $U$  
	\begin{align*}
		\mathcal{S}(\mathbb{R}_{\theta}^{d})=\left\{\mathbf{x}\in \mathbb{R}_{\theta}^{d}: \mathbf{x}=U(f)=  \int_{\mathbb{R}^d}f(t)U(t)dt,~\mbox{for some}~f\in \mathcal{S}(\mathbb{R}^d)\right\} 
	\end{align*}
	is called the  Schwartz space. 	If $\mathbf{x}$ is given by   $\mathbf{x}=U(f)\in \mathcal{S}(\mathbb{R}_{\theta}^{d})$, we define   
	\begin{align*}
		\tau_{\theta}(\mathbf{x})=f(0).
	\end{align*}
	Then $\tau_{\theta}$ can be extended to a normal semifinite faithful trace on $\mathbb{R}_{\theta}^{d}$ and $L_{p}(\mathbb{R}_{\theta}^{d})$ is the noncommutative $L_{p}$ space associated to pairs $(\mathbb{R}_{\theta}^{d},\tau_{\theta})$ with the norm   $\|\mathbf{x}\|_{L_{p}(\mathbb{R}_{\theta}^{d})}=(\tau_{\theta}(|\mathbf{x}|^p))^{\frac{1}{p}}$. Here we abuse the notation $\tau_{\theta}$, which is the same as that of quantum tori. One can easily distinguish them from context. The  Schwartz space $\mathcal{S}(\mathbb{R}_{\theta}^{d})$ 
	is dense in  $L_{p}(\mathbb{R}_{\theta}^{d})$ for $1\leq p<\infty$ with  respect to the norm and dense in $\mathbb{R}_{\theta}^{d}$ in weak$^{*}$ topology, see \cite{MR4156216} for more information.

	Then we define the noncommutative Riesz transforms in this setting. Let $\mathbf{x}=U(f)\in \mathcal{S}(\mathbb{R}_{\theta}^{d})$.  Since there are no underlying points,  it is convenient to understand the Riesz transform as the Fourier multiplier form.  Note that for $t\in\mathbb{R}^d$, $U(t)$ plays the role of $e^{2\pi \mathrm{i}\langle  t,\cdot \rangle}$ on $\mathbb{R}^d$. Then $\mathbf{x}=U(f)$ can be viewed as the inverse Fourier transform of $f$. Inspired by the classical formula: 
	$$	R_{j}\check{f}(x)=\int_{\mathbb{R}^d}f(\xi)e^{2\pi \mathrm{i} \langle x,\xi  \rangle}(-\mathrm{i})\frac{\xi_{j}}{|\xi|}d\xi,$$
	we define the $j$-th noncommutative Riesz transform of $\mathbf{x}=U(f)$ in the frequency space of $f$ with the    Fourier multiplier form, i.e., 
	\begin{align*}
		R_{j}\mathbf{x}=\int_{\mathbb{R}^d}f(\xi)U(\xi)(-\mathrm{i})\frac{\xi_{j}}{|\xi|}d\xi.
	\end{align*}
	Following  the same method, for any $\varepsilon>0$,   the $j$-th noncommutative truncated Riesz transform of $\mathbf{x}=U(f)$ is defined by 
	\begin{align*}
		R_{j}^{\varepsilon}\mathbf{x}=\int_{\mathbb{R}^d}f(\xi)U(\xi)(-\mathrm{i})\frac{\xi_{j}}{|\xi|}m(\varepsilon|\xi|)d\xi,
	\end{align*} 
	where we use the  associated multiplier   given  in (\ref{24329.4}).
	For   the sequence $(R_{j}^{\varepsilon}\mathbf{x})_{\varepsilon}$,  its    $L_{p}( {\mathbb{R}_{\theta}^{d}};l_{\infty}(\mathbb{R}_{+}))$ norm is   denoted by $\|\sup_{\varepsilon}R_{j}^{\varepsilon}\mathbf{x}\|_{L_{p}(\mathbb{R}_{\theta}^{d})}$. Similar to the quantum tori, $\sup_{\varepsilon}R_{j}^{\varepsilon}\mathbf{x}$ is just a notation and  does not make sense.

	Now, we  state our results on quantum Euclidean space,   the boundedness of    the maximal Riesz transforms   in terms of the Riesz transforms on $L_{p}(\mathbb{R}_{\theta}^{d})$.
	\begin{theorem}\label{24521.3}
		Let $1<p\leq\infty$ and $\mathbf{x}\in {L_{p}(\mathbb{R}^{d}_{\theta})}$. Then there exists a constant $C_{d,p}$ such that 
		\begin{align*}
			\|\sup_{\varepsilon>0}R_{j}^{\varepsilon}\mathbf{x}\|_{L_{p}(\mathbb{R}^{d}_{\theta})}\leq C_{d,p}\|R_{j}\mathbf{x}\|_{L_{p}(\mathbb{R}^{d}_{\theta})}.
		\end{align*}
	\end{theorem}
	
	Similar to Theorem \ref{241117.3}, we also get the dimension-free result on quantum Euclidean space.

	\begin{theorem}\label{24728.1}
		Let  $2\leq p<\infty$ and $\mathbf{x}\in {L_{p}(\mathbb{R}^{d}_{\theta})}$. Then  we have 
		\begin{align*}
			\|\sup_{\varepsilon>0}R_{j}^{\varepsilon}\mathbf{x}\|_{L_{p}(\mathbb{R}^{d}_{\theta})}\leq (96+2\sqrt{2})^{\frac{2}{p}}\|R_{j}\mathbf{x}\|_{L_{p}(\mathbb{R}^{d}_{\theta})}.
		\end{align*}
	\end{theorem}

	The rest of this paper is organized as follows.   In Section \ref{24521.1},   we 
	present some preliminaries and notions on  the noncommutative  $L_{p}$ space,   quantum tori, and quantum Euclidean space.  This section also introduces the transference method on quantum tori and  calculus on quantum Euclidean space. In Section \ref{241121.3}, we    investigate  the boundedness  of the maximal Riesz transforms in terms of the Riesz transforms on   quantum tori, via  the transference technique to study the corresponding one  in the setting of  tensor von Neumann algebra $ L_{\infty}(\mathbb{T}^d)\overline{\otimes}\mathcal{M}$ or $ L_{\infty}(\mathbb{R}^d)\overline{\otimes}\mathcal{M}$. Section \ref{24913.1} is   aim to investigate these results on quantum Euclidean space relying on  the noncommutative Calder\'{o}n's  transference principle (see \cite[Theorem 3.1]{MR4202493}).

	\vspace{0.2cm}
	
	\noindent\textbf{Notation.} Throughout this paper, we only consider the dimension $d\geq2$. The letter $C$ stands for a
	positive constant,  independent of the variables, not necessarily the same one in each occurrence. Notation $C_{p,d}$ means a positive constant depending on the parameters $p$ and $d$. $X\lesssim Y$  means $X\leq C Y$ for some positive constant $C$. For any $1\leq p\leq\infty$, we denote  $p'$ as the conjugate index of $p$, i.e., $\frac{1}{p}+\frac{1}{p'}=1$.

	\section{Preliminaries}\label{24521.1}
	\noindent\textbf{2.1. Noncommutative $L_{p}$ space.} Let $\mathcal{M}$ be a von Neumann algebra equipped with a normal semifinite faithful trace $\tau$. The support of $x\in \mathcal{M}$, denoted by  $s(x)$,  is the smallest projection $e$ satisfying $exe=x$. Let $\mathcal{S}_{+}(\mathcal{M})$ be the set of all positive elements $x$ in $\mathcal{M}$ with $\tau(s(x)) <\infty$. $\mathcal{S}(\mathcal{M})$ represents  the linear span of $\mathcal{S}_{+}(\mathcal{M})$. Given $1\leq p<\infty$, we
	define
	\begin{align*}
		\|x\|_{L_{p}(\mathcal{M})}=(\tau(|x|^p))^{\frac{1}{p}},~~x\in \mathcal{S}(\mathcal{M}),
	\end{align*}
	where $|x|=(x^{*}x)^{\frac{1}{2}}$ is the modulus of $x$. Then $(\mathcal{S}(\mathcal{M}),\|\cdot\|_{L_{p}(\mathcal{M})})$ is a normed space, whose
	completion is called as the noncommutative $L_{p}$ space associated with $(\mathcal{M}, \tau )$, denoted  
	by $L_{p}(\mathcal{M})$. For convenience, we define $L_{\infty}(\mathcal{M})=\mathcal{M}$ equipped with the operator norm $\|\cdot\|_{\infty}$. Let $L_{p}(\mathcal{M})_{+}$ denote the positive part of $L_{p}(\mathcal{M})$.

	We introduce the semi-commutative space  $L_p(\mathcal{ {N}})$   associated with  the tensor von Neumann algebra $  \mathcal{ {N}}=L_{\infty}(\mathbb{R}^d)\overline{\otimes}\mathcal{M}$ equipped with its tensor trace $ \int_{\mathbb{R}^d}dx\otimes\tau$, which is  frequently used in Section \ref{241121.3}. Notice that   $L_p(\mathcal{ {N}})$ isometrically coincides with $L_p(\mathbb{R}^d;L_{p}(\mathcal{ {M}}))$, the usual $L_{p}$ space of $p$-integrable functions from $\mathbb{R}^d$ to $L_{p}(\mathcal{ {M}})$. Let $1<p<\infty$,  the following operator-valued version of the H\"{o}lder inequality (see \cite[Lemma 2.4]{hong2022noncommutative}) will be  frequently used,  
	\begin{align}\label{3802}
		\int_{\mathbb{R}^d}f(x)g(x)dx\leq \bigg(\int_{\mathbb{R}^d}f(x)^{p}dx\bigg)^{\frac{1}{p}}\bigg(\int_{\mathbb{R}^d}g(x)^{p'}dx\bigg)^{\frac{1}{p'}},
	\end{align}
	where $f:\mathbb{R}^d\rightarrow L_{1}(\mathcal{M})+L_{\infty}(\mathcal{M})$  and $g:\mathbb{R}^d\rightarrow \mathbb{C}$ are positive functions  such that  all the integrals in (\ref{3802}) make sense. Here $\leq$ is understood as the partial order in the positive cone of $\mathcal{ {M}}$. The operator-valued
	version of the Plancherel formula is also essential for us. For sufficiently nice function $f:\mathbb{R}^d\rightarrow L_{1}(\mathcal{M})+L_{\infty}(\mathcal{M})$, such as $f\in L_{\infty}(\mathbb{R}^d)\otimes L_{2}(\mathcal{M})$, we have
	\begin{align*}
		\int_{\mathbb{R}^d}|f(x)|^2dx=	\int_{\mathbb{R}^d}|\widehat{f}(\xi)|^2d\xi.
	\end{align*}
	
	\noindent\textbf{2.2. Noncommutative maximal function.} As mentioned in the introduction, a fundamental object  of this paper is the noncommutative maximal functions. We adopt its definition   introduced by  Pisier \cite{MR1648908}  and   Junge  \cite{MR1916654}.
	\begin{definition}\label{24926.1}
		Given $1\leq p\leq\infty$, $L_{p}(\mathcal{M};\,l_{\infty})$ is the space of all sequences $x=(x_{n})_{n\in\mathbb{Z}}$ in $L_{p}(\mathcal{M})$ which admit factorizations of the following form: there are $a,b\in L_{2p}(\mathcal{M}) $ and a bounded sequence $y=(y_{n})_{n\in\mathbb{Z}}$ in $L_{\infty}(\mathcal{M})$ such that $x_{n}=ay_{n}b$ for all $n\in\mathbb{Z}$. The norm of $x$ in $L_{p}(\mathcal{M};\,l_{\infty})$ is given by
		\begin{align*}
			\|x\|_{L_{p}(\mathcal{M};\,l_{\infty})}=\inf\left\{\|a\|_{L_{2p}(\mathcal{M})}\sup_{ n\in\mathbb{Z}}\|y_{n}\|_{\infty}\|b\|_{L_{2p}(\mathcal{M})}\right\},
		\end{align*}
		where the infimum is taken over all factorizations of $x= (x_{n})_{n\in\mathbb{Z}}=(ay_{n}b)_{n\in\mathbb{Z}}$ as above.
	\end{definition}
	
	To understand above definition well, let $x=(x_{n})_{n\in\mathbb{Z}}$ be a sequence of selfadjoint operators in $L_{p}(\mathcal{M})$. It was shown in \cite[Remark 4]{MR3079331} that $x\in L_{p}(\mathcal{M};\,l_{\infty})$
	if and only if there is a positive operator $a\in L_{p}(\mathcal{M})$ such that $-a \leq  x_{n}\leq a $  for all $n\in\mathbb{Z}$, and moreover,
	\begin{align}\label{241014.5}
		\|x\|_{L_{p}(\mathcal{M};\,l_{\infty})}=\inf\left\{\|a\|_{L_{p}(\mathcal{M})}:~a\in L_{p}(\mathcal{M})_{+},~-a \leq  x_{n}\leq a, ~\forall \,n\in\mathbb{Z} \right\}.
	\end{align}
	
	More generally, if  $\Lambda$  is an arbitrary index set,   $L_{p}(\mathcal{M};\,l_{\infty}(\Lambda))$ is defined by the space of all sequences $x=(x_{\lambda})_{\lambda\in \Lambda}$ in $L_{p}(\mathcal{M})$ which admit   factorizations of the following form: there are $a,b\in L_{2p}(\mathcal{M}) $ and a bounded sequence $y=(y_{\lambda})_{\lambda\in \Lambda}\subset L_{\infty}(\mathcal{M})$ such that $x_{\lambda}=ay_{\lambda}b$ for all $\lambda\in \Lambda$. The norm of $x$ in $L_{p}(\mathcal{M};\,l_{\infty}(\Lambda))$ is given by
	\begin{align*}
		\|x\|_{L_{p}(\mathcal{M};\,l_{\infty}(\Lambda))}=\inf_{x_{\lambda}=ay_{\lambda}b}\left\{\|a\|_{L_{2p}(\mathcal{M})}\sup_{ \lambda\in \Lambda}\|y_{\lambda}\|_{\infty}\|b\|_{L_{2p}(\mathcal{M})}\right\}.
	\end{align*} 
	If $x=(x_{\lambda})_{\lambda\in \Lambda}$ is a sequence of selfadjoint operators in $L_{p}(\mathcal{M})$, then $\|x\|_{L_{p}(\mathcal{M};\,l_{\infty}(\Lambda))}$ has the similar property as (\ref{241014.5}), i.e.,
	\begin{align}\label{24829.2}
		\|x\|_{L_{p}(\mathcal{M};\,l_{\infty}(\Lambda))}=\inf\left\{\|a\|_{L_{p}(\mathcal{M})}:~a\in L_{p}(\mathcal{M})_{+},~\-a \leq x_{\lambda}\leq a, ~\forall \,\lambda\in \Lambda \right\}.
	\end{align}
	The norm $\|x\|_{L_{p}(\mathcal{M};\,l_{\infty}(\Lambda))}$ is conventionally denoted by $\|\sup_{\lambda\in \Lambda} x_{\lambda}\|_{L_{p}(\mathcal{M})}$.
	It was shown in \cite[Proposition 2.1]{MR2276775} that 	a sequence $x=(x_{\lambda})_{\lambda\in \Lambda}$ in $L_{p}(\mathcal{M})$ belongs to $L_{p}(\mathcal{M};\,l_{\infty}(\Lambda))$ if and only if
	\begin{align*}
		\sup_{J\subset \Lambda,~J \mbox{ finite}}\|\sup_{j\in J}x_{j}\|_{L_{p}(\mathcal{M})}<\infty,
	\end{align*}
	and in this case,
	\begin{align}\label{0118.1}
		\|\sup_{\lambda\in \Lambda}x_{\lambda}\|_{L_{p}(\mathcal{M})}=\sup_{J\subset \Lambda, ~J \mbox{ finite}}\|\sup_{j\in J}x_{j}\|_{L_{p}(\mathcal{M})}.
	\end{align}
	
	Now, we are able to introduce the noncommutative Hardy-Littlewood maximal inequality mentioned in  the introduction  more specifically. Recall that the Hardy-Littlewood averaging operator $M_{t}f$ is defined in (\ref{241227.2}).
	\begin{proposition}\label{241227.1}
		Let $1<p\leq\infty$ and $f\in L_{p}(\mathcal{N})$ be positive. Then, there exist a   positive operator-valued function $G\in L_{p}(\mathcal{N})$ such that 
		\begin{align*} 
			M_{t}f(x)\leq G(x),~\forall t>0, \quad\mbox{with}\quad  \|G\|_{ L_{p}(\mathcal{N})}\leq C_{d,p}\|f\|_{L_{p}(\mathcal{N})}.
		\end{align*}
	\end{proposition}

	\noindent\textbf{2.3. Quantum tori $\mathbb{T}_{\theta}^{d}$.} Suppose that $d\geq2$, and $\theta=(\theta_{k,l})_{1\leq k,l\leq d}$ is a real   antisymmetric   $d\times d$ matrix. The  $d$-dimensional noncommutative torus $\mathcal{A}_{\theta }$ is the   universal $C^{*}$-algebra generated by $d$ unitary operators $V_{1},\cdots,V_{d}$, which  satisfy 
	\begin{align}\label{241218.1}
		V_{k}V_{l}=e^{2\pi\mathrm{i}\theta_{k,l}}V_{l}V_{k}, ~~~1\leq k,l\leq d.
	\end{align}
	Given $n=(n_{1},\cdots,n_{d})\in\mathbb{Z}^d$ and $\mathbf{V}=(V_{1},\cdots,V_{d})$,  we adopt the  notation  $\mathbf{V}^{n}=V_{1}^{n_{1}}\cdots V_{d}^{n_{d}}$. The polynomial in $\mathbf{V}$ is the element 
	\begin{align}\label{241112.8}
		\mathbf{x}=\sum_{n\in\mathbb{Z}^d}\alpha_{n}\mathbf{V}^{n}, ~~\alpha_{n}\in\mathbb{C},
	\end{align}
	with a finite sum, i.e.,   $\alpha_{n}=0$ for all but finite indices $n\in\mathbb{Z}^d$.  The involution algebra of all such polynomials, denoted by $\mathcal{P}_{\theta}$, is dense in $\mathcal{A}_{\theta }$. For  any polynomial  $\mathbf{x}$ with   form (\ref{241112.8}), we define 
	\begin{align*}
		{\tau}_{\theta }(\mathbf{x})=\alpha_{\mathbf{0}},
	\end{align*}
	where $\mathbf{0}=(0,\cdots,0)$.  Then $\tau_{\theta}$   is extended to a     faithful tracial state on $\mathcal{A}_{\theta}$.   The $d$-dimensional quantum torus, denoted by  $\mathbb{T}_{\theta }^{d}$, is the weak $*$-closure of $\mathcal{A}_{\theta }$ in the GNS representation of $\tau_{\theta}$.   The state $\tau_{\theta}$    extends to a normal  faithful tracial state on $\mathbb{T}_{\theta }^{d}$. Let $L_{p}(\mathbb{T}^d_{\theta })$ be the  noncommutative $L_{p}$ space associated with pairs $(\mathbb{T}_{\theta}^{d}, \tau_{\theta})$. 
	
	\vspace{0.2cm}
	
	\noindent\textbf{Transference on $\mathbb{T}_{\theta}^{d}$.}  Recall   that  the classical $d$-dimensional tours   $\mathbb{T}^d$ is the cube $[0,1]^d$ with opposite sides identified. It is equipped with the Haar measure, denoted by $dx$. 
	
	Consider  the  von Neumann algebra $\mathcal{N}_{\theta}=L_{\infty}(\mathbb{T}^d)\overline{\otimes}\mathbb{T}^d_{\theta}$ equipped with the tensor trace $\int_{\mathbb{T}^d}dx\otimes{\tau}_{\theta}$. Let $L_{p}(\mathcal{N}_{\theta})$ be the noncommutative $L_{p}$ space associated with pairs $(\mathcal{N}_{\theta},\int_{\mathbb{T}^d}dx\otimes{\tau}_{\theta})$. It is well-known that  $L_{p}(\mathcal{N}_{\theta})$ isometrically coincides with $L_p(\mathbb{T}^d; L_{p}(\mathbb{T}^d_{\theta}))$, the usual $L_{p}$ space of $p$-integrable functions from $\mathbb{T}^d$ to $L_{p}(\mathbb{T}^d_{\theta})$.

	For every $z\in \mathbb{T}^d$,  we define $W_{z}$ by
	\begin{align*}
		W_{z}(\mathbf{V}^{n})=e^{2\pi\mathrm{i} \langle  z,n \rangle}\mathbf{V}^{n},
	\end{align*}
	which is  a positive isometry on $L_{p}(\mathbb{T}^d_{\theta})$. On the one hand, $W_{z}$ is a  homomorphism of $\mathbb{T}^d_{\theta}$, which implies that for any positive element $\mathbf{x}^{*}\mathbf{x}$, we have
	\begin{align*}
		W_{z}(\mathbf{x}^{*}\mathbf{x})=W_{z}(\mathbf{x}){^*}W_{z}(\mathbf{x})\geq0.
	\end{align*}
	Indeed,  it is clear  that $W_{z}$ is a linear  multiplicative map, which reduces to  check it preserves involution. Using the commutation relation (\ref{241218.1}), we obtain 
	\begin{align*}
		(\mathbf{V}^{n})^{*}=(V_{d}^{-1})^{n_{d}}\cdots(V_{1}^{-1})^{n_{1}}=C_{\theta,n}\mathbf{V}^{-n},
	\end{align*}
	where $C_{\theta,n}=e^{-2\pi\mathrm{i}\sum_{1\leq k<l\leq d  }\theta_{k,l}n_{k}n_{l}}$. Thus, 
	\begin{align*}
		W_{z}((\mathbf{V}^{n})^{*})=C_{\theta,n}e^{-2\pi\mathrm{i} \langle  z,n \rangle}\mathbf{V}^{-n}=(W_{z}(\mathbf{V}^{n}))^{*}.
	\end{align*}
	On other hand,  $W_{z}$ is trace preserving.  Therefore, it   extends to an isometry on $L_{p}(\mathbb{T}^d_{\theta})$ for $1<p<\infty$, i.e.,
	\begin{align}\label{241112.7}
		\|\mathbf{x}\|_{L_{p}(\mathbb{T}^d_{\theta})}=\|W_{z}\mathbf{x}\|_{L_{p}(\mathbb{T}^d_{\theta})}.
	\end{align}

	The following  transference principle is essential  for us  and has been proved in \cite[Proposition 2.1]{ MR3079331}.
	\begin{proposition}\label{241126.1}
		For any $\mathbf{x}\in L_{p}(\mathbb{T}^d_{\theta})$, the operator-valued function $\tilde{\mathbf{x}} : z\mapsto W_{z}(\mathbf{x}) $  
		is continuous from $\mathbb{T}^d$ 
		to $L_{p}(\mathbb{T}^d_{\theta})$ $($with respect to the weak $*$-topology for $p=\infty$$)$.\,\,Moreover, $\mathbf{x}\rightarrow \tilde{\mathbf{x}}$ is an isometric embedding from $L_{p}(\mathbb{T}^d_{\theta})$ to 
		$L_{p}(\mathcal{N}_{\theta})$.
	\end{proposition}

	\noindent\textbf{2.4. Quantum Euclidean space $\mathbb{R}_{\theta}^{d}$.} Let $d\geq 2$, and $\theta$ be a fixed antisymmetric $d\times d$ matrix.  $\mathbb{R}_{\theta}^{d}$ is a von Neumann algebra of $\mathcal{B}(L_{2}(\mathbb{R}^d))$ generated by   unitary operators $(U(t))_{t\in\mathbb{R}^d}$, which  act  on   $L_{2}(\mathbb{R}^d)$ as follows: 
	\begin{align*}
		(U(t)f)(r)=e^{-\frac{\mathrm{i}}{2}\langle t,\theta r  \rangle}f(r-t),~~f\in L_{2}(\mathbb{R}^d),~~r\in\mathbb{R}^d.
	\end{align*}

	Given $f\in L_{1}(\mathbb{R}^d)$,   we construct an element $U(f)\in \mathbb{R}_{\theta}^{d}$  as
	\begin{align*}
		U(f)g=\int_{\mathbb{R}^d}f(t)U(t)gdt,~~g\in L_{2}(\mathbb{R}^d).
	\end{align*}
	Note that the above $L_{2}(\mathbb{R}^d)$-valued  integral is absolutely convergent in the Bochner sense. The  triangle inequality in the Bochner sense implies that $U$ is bounded from $L_{1}(\mathbb{R}^d)$ to $ L_{\infty}(\mathbb{R}_{\theta}^{d})$.
	Furthermore,  it was shown in \cite[Proposition 2.8]{MR4156216} that, for $f\in L_{2}(\mathbb{R}^{d})$, the  plancherel formula holds, i.e., 
	\begin{align}\label{241225.2}
		\|U(f)\|_{L_{2}(\mathbb{R}^{d}_{\theta})}=\|f\|_{L_{2}(\mathbb{R}^{d})}.  
	\end{align}
	Then, the quantum   analogue of the Hausdorff-Young inequality  is an immediate consequence  of complex interpolation for the couples  $(L_{1}(\mathbb{R}^d),L_{2}(\mathbb{R}^d))$ and $(\mathbb{R}_{\theta}^{d},L_{2}(\mathbb{R}_{\theta}^{d}))$. Thus, $U$ can be viewed as the inverse Fourier transform on the quantum Euclidean space.

	In particular, we  denote Schwartz class in $\mathbb{R}_{\theta}^{d}$ as  the image of $U$ under the Schwartz  function,  i.e., 
	\begin{align*}
		\mathcal{S}(\mathbb{R}_{\theta}^{d})=\left\{\mathbf{x} \in \mathbb{R}_{\theta}^{d}:\mathbf{x}=U(f)=  \int_{\mathbb{R}^d}f(t)U(t)dt,~\mbox{for some}~f\in \mathcal{S}(\mathbb{R}^d)\right\}. 
	\end{align*}
	In this case, $U: \mathcal{S}(\mathbb{R}^d)\rightarrow \mathcal{S}(\mathbb{R}_{\theta}^{d})$ is a bijection, thus  $\mathcal{S}(\mathbb{R}_{\theta}^{d})$
	is equipped with a Fr\'{e}chet topology  induced by $U$ from $\mathcal{S}(\mathbb{R}^d)$, which admits its own canonical  Fr\'{e}chet topology.

	Let $\mathbf{x}\in \mathcal{S}(\mathbb{R}_{\theta}^{d})$ given by $\mathbf{x}=U(f)$, we define 
	\begin{align*}
		\tau_{\theta}(\mathbf{x})=f(0).
	\end{align*}
	Then $\tau_{\theta}$ can be extended to a normal semifinite faithful trace on $\mathbb{R}_{\theta}^{d}$.  The  noncommutative $L_{p}$ space
	$L_{p}(\mathbb{R}_{\theta}^{d})$ for $1\leq p<\infty$ is the completion of $N_{p}=\left\{\mathbf{x}\in \mathbb{R}_{\theta}^{d}: \tau_{\theta}(|\mathbf{x}|^p)<\infty \right\} $ with the norm   $\|\mathbf{x}\|_{L_{p}(\mathbb{R}_{\theta}^{d})}=(\tau_{\theta}(|\mathbf{x}|^p))^{\frac{1}{p}}$.
	Similar to the classical case,   $\mathcal{S}(\mathbb{R}_{\theta}^{d})$ is dense in  $L_{p}(\mathbb{R}_{\theta}^{d})$ for $1\leq p<\infty$ with respect to the norm, and dense in $\mathbb{R}_{\theta}^{d}$ in the weak$^*$ topology (see  \cite{MR4156216}).  
	
	\vspace{0.2cm}
	
	\noindent\textbf{Calculus on $\mathbb{R}_{\theta}^{d}$.}
	Let us first recall differential structure on $\mathbb{R}_{\theta}^{d}$. For $1\leq k\leq d$,  $\mathcal{D}_{k}$ is defined as the multiplication operator:
	\begin{align*}
		(\mathcal{D}_{k}g)(r)=2\pi r_{k}g(r),~r\in\mathbb{R}^d,
	\end{align*}
	acting on the domain $\mbox{dom}\mathcal{D}_{k}=\left\{g\in L_{2}(\mathbb{R}^d):g\in L_{2}(\mathbb{R}^d,r_{k}^2dr)\right\}$. Fixing $s\in\mathbb{R}^d$, the unitary generator $U(s)$ preserves $\mbox{dom} \mathcal{D}_{k} $ and following  formulas hold:
	\begin{align*}
		[\mathcal{D}_{k}, U(s)]=2\pi s_{k}U(s), \mbox{ and } e^{\mathrm{i}t\mathcal{D}_{k}}U(s)e^{-\mathrm{i}t\mathcal{D}_{k}}=e^{2\pi\mathrm{i}ts_{k}}U(s)\in\mathbb{R}_{\theta}^{d},\quad t>0.
	\end{align*}
	For general $x\in \mathbb{R}_{\theta}^{d}$, it was shown in \cite[Proposition 6.12]{MR4008371}   that if the commutator $[\mathcal{D}_{k},x]$   extends  to a bounded operator on $L_{2}(\mathbb{R}^d)$,  it can be written as
	\begin{align*}
		[\mathcal{D}_{k},x]=\lim_{t\rightarrow0}\frac{e^{\mathrm{i}t\mathcal{D}_{k}}xe^{-\mathrm{i}t\mathcal{D}_{k}}-x}{\mathrm{i}t}, 
	\end{align*}
	with respect to the strong operator topology. Then  $[\mathcal{D}_{k},x]\in \mathbb{R}_{\theta}^{d}$ is defined as the partial derivative $\partial_{k}x$ of $x\in \mathbb{R}_{\theta}^{d}$.
	The  gradient   $\nabla_{\theta}=(-\mathrm{i} \mathcal{D}_{1},\cdots, -\mathrm{i} \mathcal{D}_{d}) $  is the operator
	defined on the domain $L_{2}(\mathbb{R}^d,r_{1}^2dr)\cap\cdots\cap L_{2}(\mathbb{R}^d,r_{d}^2dr)$. Note that for $t\in\mathbb{R}^d$, $ \exp(\langle  t,\nabla_{\theta} \rangle) $ is the operator on $ L_{2}(\mathbb{R}^d)$ given by 
	\begin{align*}
		\exp(\langle  t,\nabla_{\theta} \rangle)g(r)= e^{2\pi\mathrm{i}\langle t,r  \rangle}g(r),~~r\in\mathbb{R}^d,\,\,g\in  L_{2}(\mathbb{R}^d).
	\end{align*}
	
	Now, we are ready to define the translation on $\mathbb{R}_{\theta}^{d}$. 
	\begin{definition} \label{24918.2}
		Let	$\mathbf{x}\in \mathbb{R}_{\theta}^{d}$ and $t\in \mathbb{R}^d$. The translation $T_{t}\mathbf{x}$ is defined by
		\begin{align*}
			T_{t}\mathbf{x}=\exp(\langle t,\nabla_{\theta}  \rangle)\mathbf{x}\exp(-\langle   t,\nabla_{\theta}\rangle).
		\end{align*}
	\end{definition} 
	
	\begin{remark}\label{241215.5}
		If $\theta=0$, then $T_{t}$ is   the classical translation operator   $\tilde{T}_{-t}$  defined on $\mathbb{R}^d$, which is   $\tilde{T}_{-t}g(\cdot)=g(\cdot-t)$. It was shown in \cite{MR4156216} that $T_{t}$ is an isometry in every $L_{p}(\mathbb{R}_{\theta}^{d})$, for $1\leq p\leq\infty$. Notice that
		\begin{align*}
			T_{t}\mathbf{x}=\exp(-\langle t,\nabla_{\theta}  \rangle)^*\mathbf{x}\exp(-\langle   t,\nabla_{\theta}\rangle), 
		\end{align*}
		then, $T_{t}$ is  a positive operator.
	\end{remark}

	Based on translation, it is natural to define the convolution $*_{\theta}$ on $\mathbb{R}_{\theta}^{d}$. Here we denote $*_{\theta}$ to avoid confusion with the usual convolution $*$ on $\mathbb{R}^d$.
	\begin{definition} 
		Let $1\leq p<\infty$ and  $\mathbf{x}\in L_{p}(\mathbb{R}_{\theta}^{d})$. For $\psi\in L_{1}(\mathbb{R}^{d})$, the convolution of  $\mathbf{x}$ with $\psi$ is defined by
		\begin{align}\label{24726.3}
			\psi*_{\theta}\mathbf{x}=\int_{\mathbb{R}^d}\psi(s)T_{-s} \mathbf{x}\,ds.
		\end{align}
	\end{definition}
	
	Consider $\mathbf{x}\in L_{2}(\mathbb{R}^{d}_{\theta})$ given by $\mathbf{x}=U(f)$ for $f\in L_{2}(\mathbb{R}^d)$. The following facts are  essential for us:
	\begin{align}\label{24065.3}
		T_{t}\mathbf{x}=U\big(e^{2\pi\mathrm{i}\langle t,\cdot  \rangle}f(\cdot)\big),
	\end{align}
	and
	\begin{align}\label{24726.4}
		\psi*_{\theta}\mathbf{x}=U(\widehat{\psi}f).
	\end{align}
	
	\vspace{0.05cm}
	
	\section{Maximal Riesz transform on quantum tori}\label{241121.3}
	This section is devoted to establishing the boundedness of the maximal Riesz transforms in terms of the Riesz transforms on quantum tori, i.e., Theorem \ref{241117.3}. One strategy for approaching our problem is to  transfer them to  corresponding ones in the setting of  tensor von Neumann algebra $ L_{\infty}(\mathbb{T}^d)\overline{\otimes}\mathcal{M}$ (see Proposition \ref{241118.1}).
	This  transference method is shown in \cite{MR3079331} and has been applied as a powerful tool in many works (see e.g. \cite{MR4381188,MR3778570}).

	Given $f:\mathbb{T}^d\rightarrow \mathcal{ {M}}$, its $j$-th semi-commutative Riesz transform $R_{j}f$ has the same form with (\ref{241117.8}). For  any $\varepsilon>0$, the associated  semi-commutative truncated Riesz transform  $R_{j}^{\varepsilon}f$   admits the following expression
	\begin{align}\label{25014.1}
		R_{j}^{\varepsilon}f(x)=\sum_{n\in \mathbb{Z}^d}(-\mathrm{i})\frac{n_{j}}{|n|}m(\varepsilon|n|)\widehat{f}(n)e^{2\pi\mathrm{i}\langle   n,x\rangle },
	\end{align}
	where    $m$ is given in (\ref{240612.1}).
	As for the  maximal Riesz  transform,   its $L_{p}$ norm   is understood as the  $L_{p}(L_{\infty}(\mathbb{T}^d)\overline{\otimes}\mathcal{M};l_{\infty}(\mathbb{R}_{+}))$ norm of the sequence $(R_{j}^{\varepsilon}f)_{\varepsilon}$ in the noncommutative case,  denoted by $\|\sup_{\varepsilon}R_{j}^{\varepsilon}f\|_{L_{p}(L_{\infty}(\mathbb{T}^d)\overline{\otimes}\mathcal{M})}$. The  results corresponding to Theorem \ref{241117.3} on the von Neumann algebra $ L_{\infty}(\mathbb{T}^d)\overline{\otimes}\mathcal{M}$ are    stated
	as follows.
	
	\begin{proposition}\label{241118.1}
		Let  $1<p<\infty$ and $f\in L_{p}(L_{\infty}(\mathbb{T}^d)\overline{\otimes}\mathcal{M})$, then there exists a constant $C_{d,p}$ such that 
		\begin{align}\label{241112.1}
			\|\sup_{\varepsilon>0}R_{j}^{\varepsilon}f\|_{L_{p}( L_{\infty}(\mathbb{T}^d)\overline{\otimes}\mathcal{M})}\leq C_{d,p}\|R_{j}f\|_{L_{p}(L_{\infty}(\mathbb{T}^d)\overline{\otimes}\mathcal{M})}.
		\end{align}
		Moreover, if $2\leq p<\infty$, the  constant in  $(\ref{241112.1})$ is independent of the dimension, i.e.,   
		\begin{align}\label{241112.2}
			\|\sup_{\varepsilon>0}R_{j}^{\varepsilon}f\|_{L_{p}(L_{\infty}(\mathbb{T}^d)\overline{\otimes}\mathcal{M})}\leq (96+2\sqrt{2})^{\frac{2}{p}}\|R_{j}f\|_{L_{p}(L_{\infty}(\mathbb{T}^d)\overline{\otimes}\mathcal{M})}.
		\end{align}
	\end{proposition}     
	
	Now, let us   use Proposition  \ref{241126.1} and Proposition \ref{241118.1}, whose proof will be provided later, to prove   Theorem \ref{241117.3}.
	
	\vspace{0.2cm}
	
	\noindent\textbf{Proof of Theorem \ref{241117.3}}: By  density arguments, it suffices to prove Theorem \ref{241117.3} for a polynomial  $\mathbf{x}=\sum_{n\in\mathbb{Z}^d}\widehat{\mathbf{x}}(n)\mathbf{V}^{n}$.
	We have shown in Subsection 2.3  that   $W_{z}$ is a positive isometry on $L_{p}(\mathbb{T}^d_{\theta})$. Thus, it extends to an isometry on $L_{p}(\mathbb{T}^d_{\theta};l_{\infty}(\mathbb{R}_{+}))$ (see \cite[Proposition 5.3]{MR4611833}).  Hence, for $z\in\mathbb{T}^d$, 
	\begin{align}\label{241225.1}
		\|\sup_{\varepsilon>0}R_{j}^{\varepsilon}\mathbf{x}\|^{p}_{L_{p}(\mathbb{T}^d_{\theta})}=\|\sup_{\varepsilon>0}W_{z} R_{j}^{\varepsilon}\mathbf{x}\|^{p}_{L_{p}(\mathbb{T}^d_{\theta})}.
	\end{align}
	Let  $\tilde{\mathbf{x}}$  be the operator-valued function  defined in  Proposition \ref{241126.1}. By (\ref{25014.1}), we get
	\begin{align}\label{241225.5} 
		R_{j}^{\varepsilon}\tilde{\mathbf{x}}(z)=&\sum_{n\in \mathbb{Z}^d}(-\mathrm{i})\frac{n_{j}}{|n|}m(\varepsilon|n|)\widehat{\tilde{\mathbf{x}} }(n)e^{2\pi\mathrm{i} \langle  n,z \rangle}\nonumber\\
		=&\sum_{n\in \mathbb{Z}^d}(-\mathrm{i})\frac{n_{j}}{|n|}m(\varepsilon|n|)\widehat{{\mathbf{x}} }(n)\mathbf{V}^{n}e^{2\pi\mathrm{i} \langle  n,z \rangle }=W_{z}R_{j}^{\varepsilon}\mathbf{x}.
	\end{align} 
	Then, integrating  both sides of (\ref{241225.1}) over $z\in\mathbb{T}^d$, we obtain  
	\begin{align*}		\|\sup_{\varepsilon>0}R_{j}^{\varepsilon}\mathbf{x}\|^p_{L_{p}(\mathbb{T}^d_{\theta})}=\int_{\mathbb{T}^d}\|\sup_{\varepsilon>0}W_{z} R_{j}^{\varepsilon}\mathbf{x}\|^p_{L_{p}(\mathbb{T}^d_{\theta})}dz
		=\int_{\mathbb{T}^d}\|\sup_{\varepsilon>0}R_{j}^{\varepsilon}\tilde{\mathbf{x}}(z)\|^{p}_{L_{p}(\mathbb{T}^d_{\theta})}dz.
	\end{align*}
	Recall that  $\mathcal{N}_{\theta}=L_{\infty}(\mathbb{T}^d)\overline{\otimes}\mathbb{T}^d_{\theta}$ is the tensor von Neumann algebra introduced in Subsection 2.3.  Once we have 
	\begin{align}\label{241126.3}
		\int_{\mathbb{T}^d}\|\sup_{\varepsilon>0}R_{j}^{\varepsilon}\tilde{\mathbf{x}}(z)\|^{p}_{L_{p}(\mathbb{T}^d_{\theta})}dz\leq \|\sup_{\varepsilon>0}R_{j}^{\varepsilon}\tilde{\mathbf{x}}\|^p_{L_{p}( \mathcal{N}_{\theta})},
	\end{align}
	we are able to apply Proposition \ref{241118.1} in this case to get
	\begin{align*} 
		\|\sup_{\varepsilon>0}R_{j}^{\varepsilon}\tilde{\mathbf{x}}\|^{p}_{ L_{p}( \mathcal{N}_{\theta})}
		\leq C_{d,p}\|R_{j}\tilde{\mathbf{x}}\|^{p}_{ L_{p}( \mathcal{N}_{\theta})}.
	\end{align*}
	Similar to (\ref{241225.5}), the formula $R_{j}\tilde{\mathbf{x}}(z)=W_{z}R_{j}\mathbf{x}$ also holds, which implies that 
	\begin{align*}
		\|R_{j}\tilde{\mathbf{x}}\|^{p}_{ L_{p}( \mathcal{N}_{\theta})} =\int_{\mathbb{T}^d}\|R_{j}\tilde{\mathbf{x}}(z)\|^{p}_{L_{p}(\mathbb{T}^d_{\theta})}dz
		= \int_{\mathbb{T}^d}\|W_{z}R_{j}\mathbf{x}\|^{p}_{L_{p}(\mathbb{T}^d_{\theta})}dz.
	\end{align*}
	Using  (\ref{241112.7}), we obtain 
	\begin{align*}
		\int_{\mathbb{T}^d}\|W_{z}R_{j}\mathbf{x}\|^{p}_{L_{p}(\mathbb{T}^d_{\theta})}dz=\|R_{j}\mathbf{x}\|^{p}_{L_{p}(\mathbb{T}^d_{\theta})}.
	\end{align*}
	Combining the above estimates, (\ref{241112.3})  has been proved. 
	
	It remains to prove (\ref{241126.3}). According to Definition \ref{24926.1}, for any $\delta>0$, there exist $a,b\in L_{2p}( \mathcal{N}_{\theta})$ and $(y_{\varepsilon})_{\varepsilon>0}\subset \mathcal{N}_{\theta}$ such that  $R_{j}^{\varepsilon}\tilde{\mathbf{x}}=ay_{\varepsilon}b$ for all $\varepsilon>0$, with
	\begin{align*}
		\|a\|_{L_{2p}(\mathcal{N}_{\theta} )}\,\sup_{\varepsilon>0}\|y_{\varepsilon}\|_{ \mathcal{N}_{\theta}}\, \|b\|_{L_{2p}( \mathcal{N}_{\theta})}\leq \|\sup_{\varepsilon>0}R_{j}^{\varepsilon}\tilde{\mathbf{x}}\|_{L_{p}( \mathcal{N}_{\theta})}+\delta.
	\end{align*}
	Thus
	\begin{align*}
		\int_{\mathbb{T}^d}\|\sup_{\varepsilon>0}R_{j}^{\varepsilon}\tilde{\mathbf{x}}(z)\|^{p}_{L_{p}(\mathbb{T}^d_{\theta})}dz\leq& \int_{\mathbb{T}^d} \|a(z)\|^p_{L_{2p}(\mathbb{T}^d_{\theta})}\,\sup_{\varepsilon>0}\|y_{\varepsilon}(z)\|^p_{\mathbb{T}^d_{\theta}} \,\|b(z)\|^p_{L_{2p}(\mathbb{T}^d_{\theta})}dz\\
		\leq& \|a\|^{p}_{L_{2p}( \mathcal{N}_{\theta})}\,\sup_{\varepsilon>0}\|y_{\varepsilon}\|^{p}_{ \mathcal{N}_{\theta}}\, \|b\|^{p}_{L_{2p}( \mathcal{N}_{\theta})}\\
		\leq& \big(\|\sup_{\varepsilon>0}R_{j}^{\varepsilon}\tilde{\mathbf{x}}\|_{L_{p}( \mathcal{N}_{\theta})}+\delta\big)^p.
	\end{align*}
	The arbitrariness of   $\delta$ implies that  (\ref{241126.3}) holds.
	
	Given (\ref{241112.2}), the proof of   (\ref{241112.4}) is similar; we omit the details.
	$\hfill\square$
	
	\vspace{0.2cm}

	The rest part of this section is    to  prove  Proposition \ref{241118.1}. 
	For this purpose, we first establish $L_{p}$ boundedness of the maximal Riesz transforms in terms of the Riesz transforms  within the framework of  the tensor von Neumann algebra $  \mathcal{ {N}}=L_{\infty}(\mathbb{R}^d)\overline{\otimes}\mathcal{M}$. Then, according to the transference principle (see Lemma \ref{241113.1}), we are able to prove this proposition.

	Given a $L_{1}(\mathcal{ {M}})\,\cap\,L_{\infty}(\mathcal{ {M}})$-valued compactly supported function $f$,
	recall its   $j$-th semi-commutative Riesz transform,  still denoted as $R_{j}f$,  which admits  the following expression:
	\begin{align*}
		R_{j}f(x)=\frac{\Gamma(\frac{d+1}{2})}{\pi^{\frac{d+1}{2}}}\int_{\mathbb{R}^d}f(y)\frac{x_{j}-y_{i}}{\,|x-y|^{d+1}}dy,\quad x\notin  \mbox{supp}\,f,
	\end{align*}
	where $\mbox{supp}\,f $ is the support of $f$.
	The  semi-commutative Calder\'{o}n-Zygmund theory in \cite{MR2476951} showed that,  for any $1<p<\infty$, the Calder\'{o}n-Zygmund operator  with the kernel satisfying the size and H\"{o}rmander condition is well defined on $L_{p}(\mathcal{N})$.  The Riesz transforms
	are   typical Calder\'{o}n-Zygmund operators, which fall  into the scope of the theory in \cite{MR2476951}. The $j$-th  semi-commutative truncated Riesz transform $R_{j}^{\varepsilon}$ has the same form as (\ref{24312.1}) but acts on functions in $L_{p}(\mathcal{ {N}})$. As for the 
	maximal Riesz  transform,   its $L_{p}(\mathcal{ {N}})$ norm   is understood as the  $L_{p}(\mathcal{ {N}};l_{\infty}(\mathbb{R}_{+}))$ norm of the sequence $(R_{j}^{\varepsilon}f)_{\varepsilon}$ in the noncommutative case,  denoted by $\|\sup_{\varepsilon}R_{j}^{\varepsilon}f\|_{L_{p}(\mathcal{N})}$.

	\begin{proposition}\label{24328.3}
		Let $1<p\leq\infty$ and $f\in L_{p}(\mathcal{N})$. Then there exists a constant $C_{d,p}$ such that 
		\begin{align}\label{241126.6} 
			\|\sup_{\varepsilon>0}R_{j}^{\varepsilon}f\|_{L_{p}(\mathcal{N})}\leq C_{d,p}\|R_{j}f\|_{L_{p}(\mathcal{N})}.
		\end{align}
		Moreover, if $2\leq p<\infty$, the constant in $(\ref{241126.6})$ is independent of the dimension, i.e.,
		\begin{align}\label{24327.1}
			\|\sup_{\varepsilon>0}R_{j}^{\varepsilon}f\|_{L_{p}(\mathcal{N})}\leq (2\sqrt{2}+96)^{\frac{2}{p}}\|R_{j}f\|_{L_{p}(\mathcal{N})}.
		\end{align}
	\end{proposition}

	We first prove (\ref{241126.6}) in Proposition \ref{24328.3}. One important ingredient to prove (\ref{241126.6})  is that  there exists a real function $h$ such that 
	\begin{align}\label{24829.1}
		K_{j}^{1}(x)=R_{j}h(x),~~x\in\mathbb{R}^d,~ 1\leq j\leq d,
	\end{align}
	where $K_{j}^{1}(x)$ is the kernel of the  $j$-th  truncated Riesz transform with $\varepsilon=1$ (see \cite[Lemma 4]{MR2280788}). 
	Moreover,  let $B$ denote  the unit ball in $\mathbb{R}^d$, and $2B$  stands  for the ball with the same center as $B$   and radius  $2$,  we have $h\in L_{q}(2B)$ for $1<q<\infty$ and  $|h(x)|\leq  \frac{C}{|x|^{d+1}}$ if $|x|\geq 2$. The reader is   referred to \cite[Lemma 4]{MR2280788} for more details.
	
	\vspace{0.2cm}
	
	\noindent\textbf{Proof of  Proposition \ref{24328.3}: (\ref{241126.6}).} Let $1<p\leq\infty$ and $f\in L_{p}(\mathcal{N})$. For $\alpha\neq0\in \mathbb{R}$ and $\beta\in \mathbb{R}^d$, we define $\eta_{\alpha}$ as  $\eta_{\alpha}f(x)=f(\alpha x)$  and $\tilde{T}_{\beta}$ as  $\tilde{T}_{\beta}f(x)=f(x+\beta)$, respectively. Similar to commutative case,  the commutativity property  of the Riesz transforms with  $\tilde{T}_{\beta}$ and $\eta_{\alpha}$  also holds, 
	i.e.,
	\begin{align}\label{24115.1}
		\tilde{T}_{\beta}R_{j}=R_{j}\tilde{T}_{\beta}\quad \mbox{and} \quad	\eta_{\alpha}R_{j}=R_{j}\eta_{\alpha}.
	\end{align}
	Consider the truncated Riesz transforms. It is easy to see that $	\tilde{T}_{\beta}R_{j}^{\varepsilon}=R_{j}^{\varepsilon}\tilde{T}_{\beta}$ and $R_{j}^{\varepsilon}f(x)=\eta_{\frac{1}{\varepsilon}}R_{j}^{1}\eta_{\varepsilon}f(x)$. Thus, we get
	\begin{align}\label{240610.6}
		R_{j}^{\varepsilon}f(x)=\eta_{\frac{1}{\varepsilon}}R_{j}^{1}\eta_{\varepsilon}f(x)=\tilde{T}_{\frac{x}{\varepsilon}}R_{j}^{1}\eta_{\varepsilon}f(0)=R_{j}^{1}\tilde{T}_{\frac{x}{\varepsilon}}\eta_{\varepsilon}f(0).
	\end{align}
	Apply   (\ref{24829.1}) and (\ref{240610.6})  to obtain
	\begin{align*}
		R_{j}^{\varepsilon}f(x)=& \int_{\mathbb{R}^d}-K_{j}^{1}(y)\tilde{T}_{\frac{x}{\varepsilon}}\eta_{\varepsilon}f(y)dy\\
		=&\int_{\mathbb{R}^d}-R_{j}h(y)\tilde{T}_{\frac{x}{\varepsilon}}\eta_{\varepsilon}f(y)dy\\
		=&\int_{\mathbb{R}^d}h(y)R_{j}\tilde{T}_{\frac{x}{\varepsilon}}\eta_{\varepsilon}f(y)dy.
	\end{align*}
	In the following, we decompose $R_{j}\tilde{T}_{\frac{x}{\varepsilon}}\eta_{\varepsilon}f(y)$   into the linear  combination of four positive parts:  
	\begin{align*} 
		R_{j}\tilde{T}_{\frac{x}{\varepsilon}}\eta_{\varepsilon}f(y)=&\Re^{+}(R_{j}\tilde{T}_{\frac{x}{\varepsilon}}\eta_{\varepsilon}f(y))+\Re^{-}(R_{j}\tilde{T}_{\frac{x}{\varepsilon}}\eta_{\varepsilon}f(y))\\
		&+\mathrm{i}\Im^{+}(R_{j}\tilde{T}_{\frac{x}{\varepsilon}}\eta_{\varepsilon}f(y))+\mathrm{i}\Im^{-}(R_{j}\tilde{T}_{\frac{x}{\varepsilon}}\eta_{\varepsilon}f(y)).
	\end{align*}
	Then, the triangle inequality and   above decomposition imply that 
	$\|\sup_{\varepsilon>0}R_{j}^{\varepsilon}f\|_{L_{p}(\mathcal{N})}$ is controlled by the sum of the following four terms:
	\begin{align*}
		\uppercase\expandafter{\romannumeral1}_{1}=\bigg\|\sup_{\varepsilon>0}\int_{\mathbb{R}^d}h(y)\Re^{+}(R_{j}\tilde{T}_{\frac{x}{\varepsilon}}\eta_{\varepsilon}f(y))dy\bigg\|_{L_{p}(\mathcal{N})},\\
		\uppercase\expandafter{\romannumeral1}_{2}=\bigg\|\sup_{\varepsilon>0}\int_{\mathbb{R}^d}h(y)\Re^{-}(R_{j}\tilde{T}_{\frac{x}{\varepsilon}}\eta_{\varepsilon}f(y))dy\bigg\|_{L_{p}(\mathcal{N})},\\
		\uppercase\expandafter{\romannumeral1}_{3}=\bigg\|\sup_{\varepsilon>0}\int_{\mathbb{R}^d}h(y)\Im^{+}(R_{j}\tilde{T}_{\frac{x}{\varepsilon}}\eta_{\varepsilon}f(y))dy\bigg\|_{L_{p}(\mathcal{N})},\\
		\uppercase\expandafter{\romannumeral1}_{4}=\bigg\|\sup_{\varepsilon>0}\int_{\mathbb{R}^d}h(y)\Im^{-}(R_{j}\tilde{T}_{\frac{x}{\varepsilon}}\eta_{\varepsilon}f(y))dy\bigg\|_{L_{p}(\mathcal{N})}.
	\end{align*}
	It suffices to show that
	\begin{align}\label{24065.1}
		\max\left\{\uppercase\expandafter{\romannumeral1}_{1},\uppercase\expandafter{\romannumeral1}_{2},\uppercase\expandafter{\romannumeral1}_{3},\uppercase\expandafter{\romannumeral1}_{4} \right\} \leq C \|R_{j}f\|_{L_{p}(\mathcal{N})}.
	\end{align}
	
	We only give an estimate of $\uppercase\expandafter{\romannumeral1}_{1}$, since the estimates of $\uppercase\expandafter{\romannumeral1}_{2}$,  $\uppercase\expandafter{\romannumeral1}_{3}$,  and $\uppercase\expandafter{\romannumeral1}_{4}$ are similar.
	Let us split the integral in $\uppercase\expandafter{\romannumeral1}_{1}$ into $ 2B $ and $(2B)^{c}$,  then $\uppercase\expandafter{\romannumeral1}_{1}$ is bounded by
	\begin{align}\label{241014.1}
		\bigg\|\sup_{\varepsilon>0}\int_{2B}h(y)\Re^{+}(R_{j}\tilde{T}_{\frac{x}{\varepsilon}}\eta_{\varepsilon}f(y))dy\bigg\|_{L_{p}(\mathcal{N})}+\bigg\|\sup_{\varepsilon>0}\int_{(2B)^{c}}h(y)\Re^{+}(R_{j}\tilde{T}_{\frac{x}{\varepsilon}}\eta_{\varepsilon}f(y))dy\bigg\|_{L_{p}(\mathcal{N})}.
	\end{align}
	We first deal with the first term of  (\ref{241014.1}). Fix $s$ such that $1<s<p$. Using the H\"{o}lder inequality (\ref{3802}) and the fact that $h\in L_{s}(2B)$, we have 
	\begin{align*}
		\int_{2B}h(y)\Re^{+}(R_{j}\tilde{T}_{\frac{x}{\varepsilon}}\eta_{\varepsilon}f(y))dy\leq& \int_{2B}|h(y)|\Re^{+}(R_{j}\tilde{T}_{\frac{x}{\varepsilon}}\eta_{\varepsilon}f(y))dy\\
		\leq&\bigg(\int_{2B}|h(y)|^{s'}dy\bigg)^{\frac{1}{s'}}\bigg(\int_{2B}\big(\Re^{+}(R_{j}\tilde{T}_{\frac{x}{\varepsilon}}\eta_{\varepsilon}f(y))\big)^{s}dy\bigg)^{\frac{1}{s}}\\
		\leq& C_{s,d}\bigg(\frac{1}{|2B|}\int_{2B}\big(\Re^{+}(R_{j}\tilde{T}_{\frac{x}{\varepsilon}}\eta_{\varepsilon}f(y)))^{s}dy\bigg)^{\frac{1}{s}}.
	\end{align*}
	By (\ref{24115.1}), the above integral is bounded by
	\begin{align*}
		\bigg(\frac{1}{|2B|}\int_{2B}\big(\Re^{+}(R_{j}\tilde{T}_{\frac{x}{\varepsilon}}\eta_{\varepsilon}f(y))\big)^{s}dy\bigg)^{\frac{1}{s}}
		=&\bigg(\frac{1}{|2B|}\int_{2B}\bigg(\Re^{+}\bigg(R_{j}\eta_{\varepsilon}f\bigg(\frac{x}{\varepsilon}+y\bigg)\bigg)\bigg)^{s}dy\bigg)^{\frac{1}{s}}\\
		=&\bigg(\frac{1}{|2B|}\int_{2B}\big(\Re^{+}(R_{j}f(x+\varepsilon y))\big)^{s}dy\bigg)^{\frac{1}{s}}\\
		=&\bigg(\frac{1}{|2\varepsilon B|}\int_{2\varepsilon B}\big(\Re^{+}(R_{j}f(x+y))\big)^{s}dy\bigg)^{\frac{1}{s}}.
	\end{align*}
	Applying Proposition \ref{241227.1}  to $\big(\Re^{+}(R_{j}f)\big)^{s}\in L_{\frac{p}{s}}(\mathcal{N})$, there exists a positive operator-valued function $F\in L_{\frac{p}{s}}(\mathcal{N})$ such that  
	\begin{align*}
		\frac{1}{|2\varepsilon B|}\int_{2\varepsilon B}\big(\Re^{+}(R_{j}f(x+y))\big)^{s}dy\leq F(x),~\forall\varepsilon>0,
	\end{align*}
	with
	\begin{align*}
		\|F\|_{L_{\frac{p}{s}}(\mathcal{N})}\leq C_{d}\frac{p^2}{(p-s)^2} \big\|\big(\Re^{+}(R_{j}f)\big)^{s}\big\|_{L_{\frac{p}{s}}(\mathcal{N})}.
	\end{align*}
	For $0<\frac{1}{s}<1$, the monotone increasing property of $a\rightarrow a^{\frac{1}{s}}\,(a>0)$ leads to 
	\begin{align*}
		\bigg(\frac{1}{|2\varepsilon B|}\int_{2\varepsilon B}\big(\Re^{+}(R_{j}f(x+y))\big)^{s}dy\bigg)^{\frac{1}{s}}\leq F^{\frac{1}{s}}(x),~~\forall\varepsilon>0,
	\end{align*}
	with 
	\begin{align}\label{241215.1}
		\|F^{\frac{1}{s}}\|_{L_{p}(\mathcal{N})}=&\|F\|^{\frac{1}{s}}_{L_{\frac{p}{s}}(\mathcal{N})}\nonumber
		\leq C_{d}^{\frac{1}{s}}\frac{p^{\frac{2}{s}}}{(p-s)^{\frac{2}{s}}}\big\|\big(\Re^{+}(R_{j}f))^{s}\big\|^{\frac{1}{s}}_{L_{\frac{p}{s}}(\mathcal{N})}\nonumber\\ 
		=&C_{d}^{\frac{1}{s}}\frac{p^{\frac{2}{s}}}{(p-s)^{\frac{2}{s}}}\|\Re^{+}(R_{j}f)\|_{L_{p}(\mathcal{N})}.
	\end{align} 
	From the above arguments, we conclude that 
	\begin{align}\label{241215.2}
		-C_{p,d}F^{\frac{1}{s}}(x)\leq\int_{2B}h(y)\Re^{+}(R_{j}\tilde{T}_{\frac{x}{\varepsilon}}\eta_{\varepsilon}f(y))dy\leq C_{p,d} F^{\frac{1}{s}}(x),~~~\forall\varepsilon>0.
	\end{align}
	Finally, together with (\ref{24829.2}), (\ref{241215.1}) and (\ref{241215.2}), we see that
	\begin{align}\label{24917.1}
		\bigg\|\sup_{\varepsilon}\int_{2B}h(y)\Re^{+}(R_{j}\tilde{T}_{\frac{x}{\varepsilon}}\eta_{\varepsilon}f(y))dy\bigg\|_{L_{p}(\mathcal{N})} 
		\leq  C_{p,d}\|\Re^{+}(R_{j}f)\|_{L_{p}(\mathcal{N})} 
		\leq  C_{p,d}\| R_{j}f\|_{L_{p}(\mathcal{N})}.
	\end{align}
	
	Now, we consider  the second term of (\ref{241014.1}). The fact that $|h(x)|\leq \frac{C}{|x|^{d+1}}$ for $|x|\geq 2$  implies 
	\begin{align*}
		\int_{(2B)^{c}}h(y)\Re^{+}(R_{j}\tilde{T}_{\frac{x}{\varepsilon}}\eta_{\varepsilon}f(y))dy\leq &	\int_{(2B)^{c}}|h(y)|\Re^{+}(R_{j}\tilde{T}_{\frac{x}{\varepsilon}}\eta_{\varepsilon}f(y))dy\\
		\leq&C\int_{(2B)^{c}}\frac{1}{|y|^{d+1}}\Re^{+}(R_{j}\tilde{T}_{\frac{x}{\varepsilon}}\eta_{\varepsilon}f(y))dy\\
		=&C\sum_{k=1}^{\infty}\int_{2^{k}<|y|\leq2^{k+1}}\frac{1}{|y|^{d+1}}\Re^{+}(R_{j}\tilde{T}_{\frac{x}{\varepsilon}}\eta_{\varepsilon}f(y))dy\\
		\leq&C\sum_{k=1}^{\infty}2^{-k}\frac{1}{|2^{k+1}B|}\int_{2^{k+1}B}\Re^{+}(R_{j}\tilde{T}_{\frac{x}{\varepsilon}}\eta_{\varepsilon}f(y))dy.
	\end{align*}
	Applying (\ref{24115.1})  again, we obtain
	\begin{align*}
		\int_{(2B)^{c}}h(y)\Re^{+}(R_{j}\tilde{T}_{\frac{x}{\varepsilon}}\eta_{\varepsilon}f(y))dy\leq C\sum_{k=1}^{\infty}2^{-k}\frac{1}{|2^{k+1}\varepsilon B|}\int_{2^{k+1}\varepsilon B}\Re^{+}(R_{j}f(x+y))dy.
	\end{align*}
	For any $k>0$, $\varepsilon>0$,   Proposition \ref{241227.1} implies that there exists a positive operator-valued function $G\in L_{p}(\mathcal{N})$ such that 
	\begin{align*}
		\frac{1}{|2^{k+1}\varepsilon B|}\int_{2^{k+1}\varepsilon B}\Re^{+}(R_{j}f(x+y))dy\leq G(x),~~\forall\varepsilon >0,~ k>0,
	\end{align*}
	with 
	\begin{align}\label{241215.3}
		\|G\|_{L_{p}(\mathcal{N})}\leq C_{d,p}  \|\Re^{+}(R_{j}f) \|_{L_{p}(\mathcal{N})}.
	\end{align}
	Then, we   get
	\begin{align}\label{241215.4}
		-G(x)\leqq -C\sum_{k=1}^{\infty}2^{-k}G(x)	\leq& \int_{(2B)^{c}}h(y)\Re^{+}(R_{j}\tilde{T}_{\frac{x}{\varepsilon}}\eta_{\varepsilon}f(y))dy\nonumber\\
		\leq& C\sum_{k=1}^{\infty}2^{-k}G(x)\leqq G(x).
	\end{align}
	By (\ref{24829.2}),  (\ref{241215.3}), together with (\ref{241215.4}), we have
	\begin{align}\label{24918.1}
		\bigg\|\sup_{\varepsilon>0}\int_{(2B)^{c}}h(y)\Re^{+}(R_{j}\tilde{T}_{\frac{x}{\varepsilon}}\eta_{\varepsilon}f(y))dy\bigg\|_{L_{p}(\mathcal{N})}\leq  	C_{p,d} \|\Re^{+}(R_{j}f) \|_{L_{p}(\mathcal{N})} 
		\leq  C_{p,d}  \|R_{j}f \|_{L_{p}(\mathcal{N})}.
	\end{align}
	Then (\ref{24917.1}) and (\ref{24918.1}) show $\uppercase\expandafter{\romannumeral1}_{1} \leq C\|R_{j}f\|_{L_{p}(\mathcal{N})}.$ The same method also works for $\uppercase\expandafter{\romannumeral1}_{2}$,  $\uppercase\expandafter{\romannumeral1}_{3}$, and $\uppercase\expandafter{\romannumeral1}_{4}$, which means that (\ref{24065.1}) holds.  So (\ref{241126.6}) in Proposition \ref{24328.3} has been proved. $\hfill\square$
	
	\vspace{0.2cm}
	
	In order to give the  dimension free  estimate (\ref{24327.1}) in Proposition \ref{24328.3}, we need a factorization of the truncated Riesz transform, i.e., $R_{j}^{t}=A^{t}(R_{j}),  \mbox{ for }  1\leq j\leq d$. Here  $A^{t}$ is the multiplier operator defined by
	\begin{align}\label{24329.3}
		\widehat{A^{t}f}(\xi)=m(t|\xi|)\widehat{f}(\xi),~~f\in L_{2}(\mathcal{N}),
	\end{align}
	with the symbol function $m$ given in (\ref{240612.1}). The following lemma  on the density of $R_{j}\big(L_{p}(\mathcal{N})\big)$  in $L_{p}(\mathcal{N})$ is essential for this factorization.

	\begin{lemma}\label{24328.1}
		Let $1<p<\infty$ and $j=1,\cdots,d$. Then $R_{j}\big(L_{p}(\mathcal{N})\big)\cap (\mathcal{S}(\mathbb{R}^d)\otimes \mathcal{S}({\mathcal{M}}))$ is dense in $L_{p}(\mathcal{N})$. In particular, $R_{j}\big(L_{p}(\mathcal{N})\big)$ is dense in $L_{p}(\mathcal{N})$. 
	\end{lemma}
	\begin{proof}
		Let  $(R_{j})^2$ denote the two-fold composition of $R_{j}$. It is easy to see that	the boundedness of the operator $R_{j}$ on $L_{p}(\mathcal{N})$ shows that $(R_{j})^2\big(L_{p}(\mathcal{N})\big)\subset R_{j}\big(L_{p}(\mathcal{N})\big)$. Thus, it suffices to prove that the subset $(R_{j})^2\big(L_{p}(\mathcal{N})\big)\cap (\mathcal{S}(\mathbb{R}^d)\otimes \mathcal{S}({\mathcal{M}}))$ of $R_{j}\big(L_{p}(\mathcal{N})\big)\cap (\mathcal{S}(\mathbb{R}^d)\otimes \mathcal{S}({\mathcal{M}}))$ is dense in $L_{p}(\mathcal{N})$.
		
		In the following, let  $f\in \mathcal{S}(\mathbb{R}^d)\otimes \mathcal{S}({\mathcal{M}})$ be expressed in the form: $f=\sum_{k=1}^{n}f_{k}\otimes m_{k}$, where $f_{k}\in \mathcal{S}(\mathbb{R}^d)$ and $m_{k}\in \mathcal{S}({\mathcal{M}})$ for $1\leq k\leq n$.  For $t>0$, we define the convolution operator $T_{t}^{j}$ by
		\begin{align*}
			T_{t}^{j}f(x)=\frac{1}{\sqrt{4\pi t}}\int_{\mathbb{R}}\exp\bigg({-\frac{|x_{j}-y_{j}|^2}{4t}}\bigg)f(x_{1},\cdots,x_{j-1},y_{j},x_{j+1},\cdots,x_{d})d{y_{j}}.
		\end{align*}
		Taking the Fourier transform on both sides, we have 
		\begin{align}\label{24329.1}
			\widehat{T_{t}^{j}f}(\xi)=\exp({-4\pi^2t\,\xi_{j}^2})\widehat{f}(\xi),~~\xi\in\mathbb{R}^d.
		\end{align}
		The fact that $\mathcal{S}(\mathbb{R}^d)$  is closed under the Fourier transform ensures $\widehat{T_{t}^{j}f}\in  \mathcal{S}(\mathbb{R}^d)\otimes \mathcal{S}({\mathcal{M}}) $ if $f\in  \mathcal{S}(\mathbb{R}^d)\otimes \mathcal{S}({\mathcal{M}}) $.    Denote $\Delta=\sum_{l=1}^{d}\partial_{l}^2$, where partial derivative $\partial_{l}$, $1\leq l\leq d$, acting on $f$ is given by $\partial_{l}f=\sum_{k=1}^{n}\partial_{l}f_{k}\otimes m_{k}$. Then the identity 
		\begin{align}\label{24329.2}
			(R_{j})^2(\Delta f)=-\partial_{j}^{2}f
		\end{align}
		holds. Using the Fourier inversion formula together with (\ref{24329.1}),  we have for $t>0$,
		\begin{align*}
			T_{t}^{j}f-f=&\int_{\mathbb{R}^d}(e^{-4\pi^2t\,\xi_{j}^2}-1)\widehat{f}(\xi)e^{2\pi \mathrm{i}\langle  x,\xi \rangle}d\xi\\
			=&-\int_{\mathbb{R}^d}\bigg(\int_{0}^{t}e^{-4\pi^2s\,\xi_{j}^2}4\pi^2 \xi_{j}^2ds\bigg)\widehat{f}(\xi)e^{2\pi \mathrm{i}\langle x,\xi  \rangle}d\xi\\
			=&\int_{0}^{t}T_{s}^{j}(\partial^2_{j}f)ds.
		\end{align*}
		Submitting  (\ref{24329.2}) into the above equality, we get 
		\begin{align*}
			T_{t}^{j}f-f=&-\int_{0}^{t}T_{s}^{j}\big((R_{j})^2(\Delta f)\big)ds\\
			=&-(R_{j})^2\int_{0}^{t}T_{s}^{j}(\Delta f)ds\\
			\triangleq&-(R_{j})^2g_{t},
		\end{align*}
		where
		\begin{align*}
			g_{t}=\int_{0}^{t}T_{s}^{j}(\Delta f)ds.
		\end{align*}
		
		Based on the above argument,   we show that for any $f\in  \mathcal{S}(\mathbb{R}^d)\otimes \mathcal{S}({\mathcal{M}})  $, it is approximated arbitrarily closed  by the sequence $(R_{j})^2g_{t}$ in the $L_{p}(\mathcal{N})$ norm.  In fact, using the conclusion that $\lim_{t\rightarrow\infty}\|T^{j}_{t}g\|_{L_{p}(\mathbb{R}^d)}=0$  for $g\in \mathcal{S}(\mathbb{R}^d)$ in  \cite[Lemma 3.2]{MR4585167}, we obtain 
		\begin{align*}
			\lim_{t\rightarrow\infty}\|(R_{j})^2g_{t}-f\|_{L_{p}(\mathcal{N})}=\lim_{t\rightarrow\infty}\|T^{j}_{t}f\|_{L_{p}(\mathcal{N})}\leq \sum_{k=1}^{n}\lim_{t\rightarrow \infty}\|T^{j}_{t}f_{k}\|_{L_{p}(\mathbb{R}^d)}\|m_{k}\|_{L_{p}(\mathcal{M})}=0.
		\end{align*}
		Noticing that $g_{t}\in    \mathcal{S}(\mathbb{R}^d)\otimes \mathcal{S}({\mathcal{M}})$, which follows from $\int_{0}^{t}T_{s}^{j}(\Delta g)ds\in \mathcal{S}(\mathbb{R}^d)$ for $g\in \mathcal{S}(\mathbb{R}^d)$ (see \cite[Lemma 3.2]{MR4585167}). Then, the density of $\mathcal{S}(\mathbb{R}^d)\otimes \mathcal{S}({\mathcal{M}})$ in $L_{p}(\mathcal{N})$  completes this proof.	
	\end{proof}

	\begin{corollary}\label{24328.5}
		Let $j=1,2,\cdots,d$. Then, for any $t>0$, the $j$-th truncated Riesz transform can be factorized as 
		\begin{align}\label{24328.2}
			R_{j}^{t}f=A^{t}(R_{j}f),~~~f\in L_{2}(\mathcal{N}).
		\end{align}
		Moreover, for $1<p<\infty$, the  maximal operator for the sequence $(A^{t})_{t>0}$ is bounded from $L_{p}(\mathcal{N})$ to  $L_{p}( \mathcal{N};l_{\infty}(\mathbb{R}_{+}))$,    and the optimal constant $C_{d,p}$ in the inequality 
		\begin{align*}
			\|\sup_{t>0}R_{j}^{t}f\|_{L_{p}(\mathcal{N})}\leq C_{d,p}\|R_{j}f\|_{L_{p}(\mathcal{N})}
		\end{align*}
		is equal to the optimal constant $D_{d,p}$ in the inequality
		\begin{align*}
			\|\sup_{t>0}A^{t}f\|_{L_{p}(\mathcal{N})}\leq D_{d,p}\|f\|_{L_{p}(\mathcal{N})}.
		\end{align*}
	\end{corollary}
	\begin{proof}
		The factorization (\ref{24328.2}) follows directly from (\ref{24329.4}) and (\ref{24329.3}). 	
		
		Proposition  \ref{24328.3} and Lemma \ref{24328.1} imply that the  maximal operator  for the sequence  $(A^{t})_{t>0}$ is bounded from $L_{p}(\mathcal{N})$ to  $L_{p}( \mathcal{N};l_{\infty}(\mathbb{R}_{+}))$. Now, we turn our attention to proving $C_{d,p}=D_{d,p}$. Since Lemma \ref{24328.1} shows that $R_{j}\big(L_{p}(\mathcal{N})\big)$ is dense in $L_{p}(\mathcal{N})$, it follows 
		that
		\begin{align*}
			D_{d,p}=&\sup_{f=R_{j}h\in  R_{j} (L_{p}(\mathcal{N})) }\frac{\|\sup_{t>0}A^{t}f\|_{L_{p}(\mathcal{N})}}{\|f\|_{L_{p}(\mathcal{N})}}\\
			=&\sup_{ R_{j}h\in  R_{j} (L_{p}(\mathcal{N}))  }\frac{\|\sup_{t>0}A^{t}R_{j}h\|_{L_{p}(\mathcal{N})}}{\|R_{j}h\|_{L_{p}(\mathcal{N})}}\\
			=&\sup_{R_{j}h\in  R_{j} (L_{p}(\mathcal{N})) }\frac{\|\sup_{t>0}R_{j}^{t}h\|_{L_{p}(\mathcal{N})}}{\|R_{j}h\|_{L_{p}(\mathcal{N})}}=C_{d,p},
		\end{align*}
		where we assume that $f=R_{j}h$ for a function  $h\in L_{p}(\mathcal{N})$.
	\end{proof}

	We also need the following proposition   to prove   (\ref{24327.1}) in Proposition \ref{24328.3}. For $t>0$, we denote $K_{t}(x)=t^{-d}K(\frac{x}{t})$.
	\begin{proposition}\label{24325.3}
		Let $K\in L_{1}({\mathbb{R}^d})$ and   $f\in L_{p}(\mathcal{N})$  for $1<p\leq\infty$.  Then, we have
		\begin{align*}
			\bigg\|\sup_{r>0}\frac{1}{r}\int_{0}^{r}K_{t}*fdt\bigg\|_{L_{p}(\mathcal{N})}\leq 4C_{p}\|K\|_{ L_{1}({\mathbb{R}^d})}\|f\|_{L_{p}(\mathcal{N})}.
		\end{align*}
	\end{proposition}
	\begin{proof}
		We   assume $f\in L_{p}(\mathcal{N})$ is positive and calculate 
		\begin{align*}
			\frac{1}{r}\int_{0}^{r}K_{t}*fdt=\int_{\mathbb{R}^d}K(y)\bigg(\frac{1}{r}\int_{0}^{r}f(x-ty)dt\bigg)dy.
		\end{align*}
		Here $\frac{1}{r}\int_{0}^{r}f(x-ty)dt$ is the directional Hardy-Littlewood integral operator associated with the direction $y$ in the semi-commutative case. It was shown in \cite[Lemma 6.3]{MR4585152} that,
		\begin{align}\label{24065.7}
			\bigg\|\sup_{r>0}\frac{1}{r}\int_{0}^{r}f(x-ty)dt\bigg\|_{L_{p}(\mathcal{N})}\leq C_{p}\|f\|_{L_{p}(\mathcal{N})},~~1<p\leq\infty,
		\end{align}
		where $C_{p}$ is the $L_{p}$ norm of the one-dimensional semi-commutative Hardy-Littlewood operator.
		Then, applying (\ref{24829.2}) and (\ref{24065.7}), we obtain 
		\begin{align*}
			\bigg\|\sup_{r>0}\frac{1}{r}\int_{0}^{r}K_{t}*fdt\bigg\|_{L_{p}(\mathcal{N})}=&\bigg\|\sup_{r>0}\int_{\mathbb{R}^d}K(y)\bigg(\frac{1}{r}\int_{0}^{r}f(x-ty)dt\bigg)dy\bigg\|\\
			\leq& \int_{\mathbb{R}^d}|K(y)|\,\bigg\|\sup_{r>0}\frac{1}{r}\int_{0}^{r}f(x-ty)dt\bigg\|_{L_{p}(\mathcal{N})}dy
			\\
			\leq& C_{p}\|K\|_{ L_{1}({\mathbb{R}^d})}\|f\|_{L_{p}(\mathcal{N})}.
		\end{align*}
		
		For  general $f\in L_{p}(\mathcal{N})$, by decomposing it into the linear combination of four positive parts, this proposition is proved.
	\end{proof}
	\begin{remark}\label{240620.1}
		Combining  Lemma 3.1 in \cite{MR1916654} and Theorem 3.3 in \cite{MR2327840},   for positive vector-valued function, the  $L_{2}$ norm of the one-dimensional  semi-commutative Hardy-Littlewood operator is 24. Then, we obtain
		\begin{align*}
			\bigg\|\sup_{r>0}\frac{1}{r}\int_{0}^{r}K_{t}*fdt\bigg\|_{L_{2}(\mathcal{N})}\leq 96\|K\|_{ L_{1}({\mathbb{R}^d})}\|f\|_{L_{2}(\mathcal{N})},~~~f\in L_{2}(\mathcal{N}).
		\end{align*}
	\end{remark}

	Recall that for $f\in \mathcal{S}(\mathbb{R}^d)\otimes\mathcal{S}(\mathcal{M})$, the operator  $A^{t}$ defined in (\ref{24329.3}) can be written   as the   convolution operator, i.e.,
	\begin{align}\label{25014.6}
		A^{t}f(x)=\int_{\mathbb{R}^d}m(t|\xi|)\widehat{f}(\xi)e^{2\pi \mathrm{i} \langle   x,\xi\rangle}d\xi=\phi_{t}*f(x),
	\end{align}
	where 
	\begin{align}\label{24919.1}
		\phi_{t}(x)=\int_{\mathbb{R}^d}m(t|\xi|)e^{2\pi \mathrm{i} \langle  x,\xi \rangle}d\xi.
	\end{align}
	Some   calculations for $\phi_{t}$ (see  \cite[Theorem 2 \& Theorem 4]{liu2023p}) are used in the proof of (\ref{24327.1}), and we present  them in the following lemma.
	\begin{lemma}\label{241212.4}
		Let $\phi_{t}$ be the   function given in $(\ref{24919.1})$, then 	$\|\phi_{t}\|_{ L_{1}({\mathbb{R}^d})}=1$.  Moreover, we have 
		\begin{align*}
			\sup_{\xi\in\mathbb{R}^d}\int_{0}^{\infty}\bigg|s\frac{d}{ds}\widehat{\phi_{s}}(\xi)\bigg|^2\frac{ds}{s} \leq\frac{1}{2}.
		\end{align*}  
	\end{lemma}

	Now, we are ready to prove (\ref{24327.1}) in  Proposition \ref{24328.3}.
	
	\vspace{0.2cm}
	
	\noindent\textbf{Proof of Proposition \ref{24328.3}: (\ref{24327.1}).} Corollary \ref{24328.5} implies that it is enough to show
	\begin{align}\label{240625.3}
		\|\sup_{t>0}A^{t}f\|_{L_{p}(\mathcal{N})}\leq (96+2\sqrt{2})^{\frac{2}{p}}\|f\|_{L_{p}(\mathcal{N})},~~~2\leq p<\infty.
	\end{align}
	To prove (\ref{240625.3}), we first show the case of $p=2$, i.e.,
	\begin{align}\label{240625.1}
		\|\sup_{t>0}A^{t}f\|_{L_{2}(\mathcal{N})}\leq  (96+2\sqrt{2} ) \|f\|_{L_{2}(\mathcal{N})}.
	\end{align}
	In the following, we assume that   $f\in \mathcal{S}(\mathbb{R}^d)\otimes\mathcal{S}(\mathcal{M})$.
	Applying integration by parts, we have
	\begin{align*}
		\phi_{t}*f=\frac{1}{t}\int_{0}^{t}\phi_{s}*fds+\frac{1}{t}\int_{0}^{t}s\frac{d}{ds}\phi_{s}*fds. 
	\end{align*}
	Then,   (\ref{25014.6}) and the triangle inequality imply that
	\begin{align}\label{24065.8}
		\|\sup_{t>0}A^{t}f\|_{L_{2}(\mathcal{N})}
		\leq \bigg\|\sup_{t>0}\frac{1}{t}\int_{0}^{t}\phi_{s}*fds \bigg\|_{L_{2}(\mathcal{N})}+ \,\bigg\|\sup_{t>0} \frac{1}{t}\int_{0}^{t} s\frac{d}{ds}\phi_{s}*f ds\bigg\|_{L_{2}(\mathcal{N})}.
	\end{align}
	
	By Remark \ref{240620.1}  and Lemma \ref{241212.4}, the first term on the right-hand side of (\ref{24065.8}) is controlled  by $96\|f\|_{L_{2}(\mathcal{N})}$, i.e.,
	\begin{align}\label{240621.1}
		\bigg\|\sup_{t>0}\frac{1}{t}\int_{0}^{t}\phi_{s}*fds\bigg\|_{L_{2}(\mathcal{N})}\leq 96\|f\|_{L_{2}(\mathcal{N})}.
	\end{align}
	
	Below, decomposing $s\frac{d}{ds}\phi_{s}*f$ into the linear combination of four positive parts and using the triangle inequality, the second term on the right-hand side of (\ref{24065.8}) is controlled by the sum of the following four terms:
	\begin{align*}
		\uppercase\expandafter{\romannumeral1}_{5}=\bigg\|\sup_{t>0} \frac{1}{t}\int_{0}^{t}\Re^{+}\bigg(s\frac{d}{ds}\phi_{s}*f\bigg)ds\bigg\|_{L_{2}(\mathcal{N})},\\
		\uppercase\expandafter{\romannumeral1}_{6}=\bigg\|\sup_{t>0} \frac{1}{t}\int_{0}^{t}\Re^{-}\bigg(s\frac{d}{ds}\phi_{s}*f\bigg)ds\bigg\|_{L_{2}(\mathcal{N})},\\
		\uppercase\expandafter{\romannumeral1}_{7}=\bigg\|\sup_{t>0} \frac{1}{t}\int_{0}^{t}\Im^{+}\bigg(s\frac{d}{ds}\phi_{s}*f\bigg) ds\bigg\|_{L_{2}(\mathcal{N})},\\
		\uppercase\expandafter{\romannumeral1}_{8}=\bigg\|\sup_{t>0} \frac{1}{t}\int_{0}^{t}\Im^{-}\bigg(s\frac{d}{ds}\phi_{s}*f\bigg)ds\bigg\|_{L_{2}(\mathcal{N})}.
	\end{align*}
	We first deal with $\uppercase\expandafter{\romannumeral1}_{5}$. By the H\"{o}lder inequality (\ref{3802}), we have
	\begin{align}\label{24065.10}
		\frac{1}{t}\int_{0}^{t}\Re^{+}\bigg(s\frac{d}{ds}\phi_{s}*f\bigg)ds\leq& \bigg(\frac{1}{t}\int_{0}^{t}\bigg(\Re^{+}\bigg(s\frac{d}{ds}\phi_{s}*f\bigg)\bigg)^{2}ds\bigg)^{\frac{1}{2}}\nonumber\\
		\leq&\bigg(\int_{0}^{\infty}\bigg(\Re^{+}\bigg(s\frac{d}{ds}\phi_{s}*f\bigg)\bigg)^{2}\frac{ds}{s}\bigg)^{\frac{1}{2}}.
	\end{align}
	Combining   (\ref{24829.2}) with (\ref{24065.10}), it is obvious to get 
	\begin{align*}
		\uppercase\expandafter{\romannumeral1}_{5}
		\leq\bigg(\int_{0}^{\infty}\bigg\|\Re^{+}\bigg(s\frac{d}{ds}\phi_{s}*f\bigg)\bigg\|^2_{_{L_{2}(\mathcal{N})}}\frac{ds}{s}\bigg)^{\frac{1}{2}}
		\leq\bigg(\int_{0}^{\infty}\bigg\|s\frac{d}{ds}\phi_{s}*f\bigg\|^2_{_{L_{2}(\mathcal{N})}}\frac{ds}{s}\bigg)^{\frac{1}{2}}.
	\end{align*}
	Then, using  the Plancherel formula   and Lemma \ref{241212.4}, $\uppercase\expandafter{\romannumeral1}_{5}$ is bounded by 
	\begin{align*}
		\bigg(\sup_{\xi\in\mathbb{R}^d}\int_{0}^{\infty}\bigg|s\frac{d}{ds}\widehat{\phi_{s}}(\xi)\bigg|^2\frac{ds}{s}\bigg)^{\frac{1}{2}}\|f\|_{L_{2}(\mathcal{N})}
		\leq\frac{\sqrt{2}}{2}\|f\|_{L_{2}(\mathcal{N})}.
	\end{align*}
	The same method also works for other three terms $\uppercase\expandafter{\romannumeral1}_{6}$, $\uppercase\expandafter{\romannumeral1}_{7}$, and $\uppercase\expandafter{\romannumeral1}_{8}$, which leads to
	\begin{align}\label{24065.11}
		\bigg\|\sup_{t>0}\frac{1}{t}\int_{0}^{t} s\frac{d}{ds}\phi_{s}*f ds\bigg\|_{L_{2}(\mathcal{N})}\leq 2\sqrt{2}\|f\|_{L_{2}(\mathcal{N})}.
	\end{align}
	Finally,  combining (\ref{24065.8}), (\ref{240621.1})  and (\ref{24065.11}), we prove (\ref{240625.1}).
	
	On other hand,  Lemma \ref{241212.4} and (\ref{25014.6}) imply that 
	\begin{align}\label{240625.2}
		\|\sup_{t>0}A^{t}f(x)\|_{L_{\infty}(\mathcal{N})}=\sup_{t>0}\|\phi_{t}*f(x)\|_{L_{\infty}(\mathcal{N})}\leq \|f\|_{L_{\infty}(\mathcal{N})}.
	\end{align}
	Finally, taking an interpolation between    (\ref{240625.1}) and (\ref{240625.2}),  (\ref{240625.3}) is obtained;  we complete the proof. $\hfill\square$
	
	\vspace{0.2cm}
	
	We are ready to prove Proposition \ref{241118.1} via the following transference principle (see \cite[Theorem 5.3]{1}), which reduces the maximal inequalities on $\mathbb{T}^d$ to the corresponding ones on $\mathbb{R}^d$.

	\begin{lemma}\label{241113.1}
		Let $1<p<\infty$ and $b$ be a bounded continuous function on $\mathbb{R}^d$. For $t>0$,   the multiplier operators
		\begin{align*}
			\tilde{S}_{b,t} f (x)=\sum_{n\in\mathbb{Z}^d}b(tn)\widehat{f}(n)e^{2\pi\mathrm{i} \langle n,x  \rangle} \quad\mbox{and}\quad S_{b,t} g (x)=\int_{\mathbb{R}^d}b(t\xi)\widehat{g}(\xi)e^{2\pi\mathrm{i} \langle   \xi,x\rangle}d\xi
		\end{align*}
		are defined on operator-valued function $f\in  C^{\infty}(\mathbb{T}^d)\otimes \mathcal{S}(\mathcal{ {M}}) $ and $g\in  C_{c}^{\infty}(\mathbb{R}^d)\otimes \mathcal{S}(\mathcal{ {M}})$.
		Suppose that there exists a constant $C$ such that 
		\begin{align}\label{241126.4}
			\|\sup_{t>0}S_{b,t} g \|_{L_{p}( \mathcal{N})}\leq C\|g\|_{L_{p}( \mathcal{N})},~~~g\in   C_{c}^{\infty}(\mathbb{R}^d)\otimes \mathcal{S}(\mathcal{ {M}}).
		\end{align}
		Then for any  $ f\in  C^{\infty}(\mathbb{T}^d)\otimes \mathcal{S}(\mathcal{ {M}}) $, the following  estimate   
		\begin{align*}
			\|\sup_{t>0}\tilde{S}_{b,t}f\|_{L_{p}(L_{\infty}(\mathbb{T}^d)\overline{\otimes}\mathcal{M})}\leq C\|f\|_{L_{p}(L_{\infty}(\mathbb{T}^d)\overline{\otimes}\mathcal{M})}  
		\end{align*}
		holds with the same constant  $C$ as in $(\ref{241126.4})$.
	\end{lemma}

	The proof of Lemma \ref{241113.1} is similar to the classical one in \cite[Theorem 4.3.12]{Graf2008}, which uses
	the duality of $L_{p}( \mathcal{N};l_{\infty})$ with $L_{p'}( \mathcal{N};l_{1})$ for $1<p<\infty$; we omit the details here.

	\vspace{0.2cm}

	\noindent\textbf{Proof of Proposition \ref{241118.1}}:
	Let $1<p<\infty$ and  $f\in L_{p}(L_{\infty}(\mathbb{T}^d)\overline{\otimes}\mathcal{M})$. We  denote 
	\begin{align*}
		\tilde{A}^{t}f(x)=\sum_{n\in\mathbb{Z}^d}m(t|n|)\widehat{f}(n)e^{2\pi\mathrm{i}\langle   n,x\rangle},
	\end{align*}
	where $m$ is the continuous function given in (\ref{240612.1}). Corollary  \ref{24328.5}  shows that the maximal operator  $(A^{t})_{t>0}$ defined by (\ref{24329.3}) is  bounded  from $L_{p}(\mathcal{N})$ to  $L_{p}( \mathcal{N};l_{\infty}(\mathbb{R}_{+}))$. Applying it to Lemma \ref{241113.1}, we  deduce  that
	\begin{align*}
		\|\sup_{t>0}\tilde{A}^{t}f\|_{ L_{p}(L_{\infty}(\mathbb{T}^d)\overline{\otimes}\mathcal{M})}\leq C_{d,p}\|f\|_{ L_{p}(L_{\infty}(\mathbb{T}^d)\overline{\otimes}\mathcal{M})}.
	\end{align*}
	For any $t>0$,  the  factorization  $R_{j}^{t}f=\tilde{A}^{t}(R_{j}f)$ also holds on von Neumann algebra $ L_{\infty}(\mathbb{T}^d)\overline{\otimes}\mathcal{M}$, which implies  that  the optimal constant $C_{d,p}$ in the inequality
	\begin{align*}
		\|\sup_{t>0}R_{j}^{t}f\|_{L_{p}(L_{\infty}(\mathbb{T}^d)\overline{\otimes}\mathcal{M})}\leq C_{d,p}\|R_{j}f\|_{L_{p}(L_{\infty}(\mathbb{T}^d)\overline{\otimes}\mathcal{M})}
	\end{align*}
	is controlled by
	the optimal constant $D_{d,p}$ in the inequality
	\begin{align*}
		\|\sup_{t>0}\tilde{A}^{t}f\|_{L_{p}(L_{\infty}(\mathbb{T}^d)\overline{\otimes}\mathcal{M})}\leq D_{d,p}\|f\|_{L_{p}(L_{\infty}(\mathbb{T}^d)\overline{\otimes}\mathcal{M})}.
	\end{align*}
	Therefore,   (\ref{241112.1}) in Proposition \ref{241118.1} is proved.  As for  (\ref{241112.2}), its  proof is similar and we  omit the details. $\hfill\square$

	\section{Maximal Riesz transform on  quantum Euclidean space }\label{24913.1}
	In this section, we study the boundedness of the maximal Riesz transforms in terms of the Riesz transforms on the quantum Euclidean space. Since the noncommutative space  $L_{p}(\mathbb{R}_{\theta}^{d})$ can not be embedded isometrically into $L_{p}(\mathbb{R}^d;L_{p}(\mathbb{R}_{\theta}^{d}))$ via a natural  isomorphic mapping,   the semi-commutative result in the previous section (see Proposition \ref{24328.3})  does not directly apply in this case.  To establish our results,  we invoke the noncommutative Calder\'{o}n's  transference principle (see \cite[Theorem 3.1]{MR4202493}) and   the amenability of $\mathbb{R}^d$ as a group, i.e., Lemma \ref{240610.2} (see \cite[Corollary 4.14]{MR961261}).
	\begin{lemma}\label{240610.2}
		For any compact subset $K\subset\mathbb{R}^d$, we have
		\begin{align*}
			\inf\left\{\frac{|K-F|}{|F|}: F\subset \mathbb{R}^d~\mbox{is compact and} ~|F|>0     \right\}=1.
		\end{align*}
	\end{lemma}
	
	\vspace{0.1cm}
	
	\noindent\textbf{Proof of Theorem \ref{24521.3}}:
	It suffices to consider  the dense  subset $\mathcal{S}(\mathbb{R}^{d}_{\theta})$. Let $\mathbf{x}=U(f)\in \mathcal{S}(\mathbb{R}^{d}_{\theta})$ for some $f\in \mathcal{S}(\mathbb{R}^{d})$. Note that  (\ref{24329.4}) and  (\ref{24829.1}) imply that $m(\varepsilon |\xi|)=\widehat{h}(\varepsilon\xi)$, which shows
	\begin{align*}
		R_{j}^{\varepsilon}\mathbf{x}
		=&\int_{\mathbb{R}^d}f(\xi)U(\xi)(-\mathrm{i})\frac{\xi_{j}}{|\xi|}\widehat{h}({\varepsilon\xi})d\xi\\
		=&\int_{\mathbb{R}^d}f(\xi)U(\xi)(-\mathrm{i})\frac{\xi_{j}}{|\xi|}\int_{\mathbb{R}^d}e^{-2\pi \mathrm{i}\langle \varepsilon \xi, x  \rangle}h(x)dxd\xi\\
		=&\int_{\mathbb{R}^d}h(x)U\bigg(e^{-2\pi \mathrm{i}\langle \varepsilon x,\cdot  \rangle}f(\cdot)(-\mathrm{i})\frac{(\cdot)_{j}}{|\cdot|}\bigg)dx\\
		=&\int_{\mathbb{R}^d} h(x)T_{- \varepsilon x}R_{j}\mathbf{x}\,dx,
	\end{align*}
	where the last equality follows from (\ref{24065.3}). Then we use the triangle inequality to arrive at
	\begin{align}\label{24731.1}
		\|\sup_{\varepsilon>0}R_{j}^{\varepsilon}\mathbf{x}\|_{L_{p}(\mathbb{R}^{d}_{\theta})}=&\bigg\|\sup_{\varepsilon>0}\int_{\mathbb{R}^d} h(x)T_{-  \varepsilon x}R_{j}\mathbf{x}\,dx \bigg\|_{L_{p}(\mathbb{R}^{d}_{\theta})}\nonumber\\
		\leq& \bigg\|\sup_{\varepsilon>0} \int_{2B} h(x)T_{- \varepsilon x}R_{j}\mathbf{x}\,dx \bigg\|_{L_{p}(\mathbb{R}^{d}_{\theta})}\nonumber\\
		&+  \sum_{k\geq1}\bigg\|\sup_{\varepsilon>0} \int_{2^{k}<|x|\leq 2^{k+1}} h(x)T_{ -  \varepsilon x}R_{j}\mathbf{x}\,dx \bigg\|_{L_{p}(\mathbb{R}^{d}_{\theta})}.
	\end{align}
	
	For the first term on the right-hand side of (\ref{24731.1}), we apply  (\ref{0118.1}) to  get
	\begin{align*}
		\bigg\|\sup_{\varepsilon>0} \int_{2B} h(x)T_{ - \varepsilon x}R_{j}\mathbf{x}\,dx \bigg\|_{L_{p}(\mathbb{R}^{d}_{\theta})}=\sup_{J\subset \mathbb{R}_{+},\, J \mbox{ finite}}	\bigg\|\sup_{\varepsilon\in J} \int_{2B} h(x)T_{ -  \varepsilon x}R_{j}\mathbf{x}\,dx \bigg\|_{L_{p}(\mathbb{R}^{d}_{\theta})}.
	\end{align*}
	In the following, $J$ will be  fixed. At the same time, we choose a compact  subset $K\subset \mathbb{R}^d$ such that $2\varepsilon B$ is contained in $K$ for all $\varepsilon\in J $. Lemma \ref{240610.2} shows that for any $\delta>0$, there exists  a compact subset $F\subset \mathbb{R}^d$ with $|F|>0$ such that $\frac{|K-F|}{|F|}\leq 1+\delta$.  Since, by Remark \ref{241215.5},   for any $t>0$, $T_{t}$ extends to an isometry on $L_{p}( \mathbb{R}^{d}_{\theta};l_{\infty}(\mathbb{R}_{+}))$ (see \cite[Proposition 5.3]{MR4611833}),   we  thus obtain 
	\begin{align*}
		\bigg\|\sup_{\varepsilon\in J} \int_{2B} h(x)T_{ -  \varepsilon x}R_{j}\mathbf{x}\,dx \bigg\|^{p}_{L_{p}(\mathbb{R}^{d}_{\theta})}=&\frac{1}{|F|}\int_{F}\bigg\|\sup_{\varepsilon\in J} \int_{2B} h(x)T_{ -  \varepsilon x}R_{j}\mathbf{x}\,dx \bigg\|^{p}_{L_{p}(\mathbb{R}^{d}_{\theta})}dt\\
		=&\frac{1}{|F|}\int_{F}\bigg\|\sup_{\varepsilon\in J} T_{-t}\int_{2B} h(x)T_{t -  \varepsilon x}R_{j}\mathbf{x}\,dx \bigg\|^{p}_{L_{p}(\mathbb{R}^{d}_{\theta})}dt\\
		=&\frac{1}{|F|}\int_{F}\bigg\|\sup_{\varepsilon\in J} \int_{2B} h(x)T_{t -  \varepsilon x}R_{j}\mathbf{x}\,dx \bigg\|^{p}_{L_{p}(\mathbb{R}^{d}_{\theta})}dt.
	\end{align*}
	Now, decompose   $T_{t -  \varepsilon x}R_{j}\mathbf{x}$ into the linear combination of four positive parts. By the triangle inequality, the above  equality is controlled by the sum of the following four terms:
	\begin{align*}
		\uppercase\expandafter{\romannumeral2}_{1}=\frac{1}{|F|}\int_{F}\bigg\|\sup_{\varepsilon\in J} \int_{2B} h(x)\Re^{+}(
		T_{t-  \varepsilon x}R_{j}\mathbf{x}) dx \bigg\|_{L_{p}(\mathbb{R}^{d}_{\theta})}^{p}dt,\\
		\uppercase\expandafter{\romannumeral2}_{2}=\frac{1}{|F|}\int_{F}\bigg\|\sup_{\varepsilon\in J} \int_{2B} h(x)\Re^{-}(
		T_{t-  \varepsilon x}R_{j}\mathbf{x}) dx \bigg\|_{L_{p}(\mathbb{R}^{d}_{\theta})}^{p}dt,\\
		\uppercase\expandafter{\romannumeral2}_{3}=\frac{1}{|F|}\int_{F}\bigg\|\sup_{\varepsilon\in J} \int_{2B} h(x)\Im^{+}(	T_{t-  \varepsilon x}R_{j}\mathbf{x}) dx \bigg\|_{L_{p}(\mathbb{R}^{d}_{\theta})}^{p}dt,\\
		\uppercase\expandafter{\romannumeral2}_{4}=\frac{1}{|F|}\int_{F}\bigg\|\sup_{\varepsilon\in J} \int_{2B} h(x)\Im^{-}(	T_{t-  \varepsilon x}R_{j}\mathbf{x}) dx \bigg\|_{L_{p}(\mathbb{R}^{d}_{\theta})}^{p}dt.
	\end{align*}
	
	We first deal with $\uppercase\expandafter{\romannumeral2}_{1}$. Fix $s$ such that $1<s<p$. The fact that $h\in L_{s}(2B)$   and the H\"{o}lder inequality (\ref{3802}) imply that
	\begin{align*}
		\int_{2B} h(x)\Re^{+}(
		T_{t-  \varepsilon x}R_{j}\mathbf{x}) dx\leq& \int_{2B} |h(x)|\Re^{+}(
		T_{t-  \varepsilon x}R_{j}\mathbf{x}) dx\\
		\leq& \bigg(\int_{2B}|h(x)|^{s'}dx\bigg)^{\frac{1}{s'}} \bigg(\int_{2B}(\Re^{+}(
		T_{t-  \varepsilon x}R_{j}\mathbf{x}))^{s }dx\bigg)^{\frac{1}{s}}\\ 
		=& C_{s,d} \bigg(\frac{1}{|2B|}\int_{2B}(\Re^{+}(
		T_{t- \varepsilon x}R_{j}\mathbf{x}))^{s }dx\bigg)^{\frac{1}{s}}\\
		=& C_{s,d}\bigg(\frac{1}{|2\varepsilon B|}\int_{t-2 \varepsilon B}(\Re^{+}(
		T_{x}R_{j}\mathbf{x}))^{s }dx\bigg)^{\frac{1}{s}}. 
	\end{align*}
	Since $t\in F$ and $\cup_{\varepsilon\in J}2\varepsilon B\subset K$, we have
	\begin{align*}
		\int_{2B} h(x)\Re^{+}(
		T_{t-  \varepsilon x}R_{j}\mathbf{x}) dx\leq C_{s,d} \bigg(\frac{1}{|2\varepsilon B|}\int_{t-2 \varepsilon B}(\Re^{+}(
		T_{x}R_{j}\mathbf{x}\,\chi_{F-K}(x)))^{s }dx\bigg)^{\frac{1}{s}}. 
	\end{align*}
	Applying  Proposition \ref{241227.1} to  the function   $(\Re^{+}(
	T_{(\cdot)}R_{j}\mathbf{x}\,\chi_{F-K}(\cdot)))^{s }\in L_{\frac{p}{s}}(\mathbb{R}^d;L_{\frac{p}{s}}(\mathbb{R}^d_{\theta}))$,   we   find a  positive operator-valued function $F_{1 }\in L_{\frac{p}{s}}(\mathbb{R}^d;L_{\frac{p}{s}}(\mathbb{R}^d_{\theta}))$ satisfying
	\begin{align*}
		\bigg(\frac{1}{|2\varepsilon B|}\int_{t-2 \varepsilon B}(\Re^{+}(
		T_{x}R_{j}\mathbf{x}\,\chi_{F-K}(x)))^{s }dx\bigg)^{\frac{1}{s}}\leq F_{1 }^{\frac{1}{s}}(t),~~\forall \varepsilon\in J,
	\end{align*}
	with
	\begin{align}\label{241215.6}
		\|F_{1 }^{\frac{1}{s}}\|_{L_{p}(\mathbb{R}^d;L_{p}(\mathbb{R}^d_{\theta}))}\leq C_{d}^{\frac{1}{s}}\frac{p^{\frac{2}{s}}}{(p-s)^{\frac{2}{s}}}\|\Re^{+}(
		T_{(\cdot)}R_{j}\mathbf{x}\,\chi_{F-K}(\cdot))\|_{L_{p}(\mathbb{R}^d;L_{p}(\mathbb{R}^d_{\theta}))}.
	\end{align}
	We  conclude that 
	\begin{align}\label{241213.1}
		-F_{1 }^{\frac{1}{s}}(t)\lesssim \int_{2B} h(x)\Re^{+}(
		T_{t-  \varepsilon x}R_{j}\mathbf{x}) dx\lesssim  F_{1 }^{\frac{1}{s}}(t),~~\forall \varepsilon\in J.
	\end{align}
	Combining  (\ref{24829.2}), (\ref{241215.6}), and (\ref{241213.1}),  we have 
	\begin{align*}
		\uppercase\expandafter{\romannumeral2}_{1}\leq\frac{1}{|F|}\|F_{1}^{\frac{1}{s}}\|^p_{L_{p}(\mathbb{R}^d;L_{p}(\mathbb{R}^d_{\theta}))}
		\lesssim &  \frac{1}{|F|}\|\Re^{+}(T_{(\cdot)}R_{j}\mathbf{x}\,\chi_{F-K}(\cdot))\|^p_{L_{p}(\mathbb{R}^d;L_{p}(\mathbb{R}^d_{\theta}))}\\
		\lesssim& \frac{1}{|F|}\|T_{(\cdot)}R_{j}\mathbf{x}\,\chi_{F-K}(\cdot)\|^p_{L_{p}(\mathbb{R}^d;L_{p}(\mathbb{R}^d_{\theta}))}\\
		\lesssim&	 \frac{|K-F|}{|F|}\|R_{j}\mathbf{x}\|^p_{L_{p}(\mathbb{R}^d_{\theta})}\\
		\lesssim &(1+\delta)\|R_{j}\mathbf{x}\|^p_{L_{p}(\mathbb{R}^d_{\theta})}.
	\end{align*}
	The same argument  also applies to $\uppercase\expandafter{\romannumeral2}_{2}$, $\uppercase\expandafter{\romannumeral2}_{3}$, and $\uppercase\expandafter{\romannumeral2}_{4}$, which concludes that
	\begin{align}\label{24065.4}
		\bigg\|\sup_{\varepsilon\in J} \int_{2B} h(x)T_{ - \varepsilon x}R_{j}\mathbf{x}\,dx \bigg\|^{p}_{L_{p}(\mathbb{R}^{d}_{\theta})}\lesssim (1+\delta)\|R_{j}\mathbf{x}\|^p_{L_{p}(\mathbb{R}^{d}_{\theta})}.
	\end{align}
	The implicit constant in  (\ref{24065.4}) is independent of $J$. Thus, the arbitrariness of $J$ and $\delta$ allows us to get 
	\begin{align}\label{24424.1}
		\bigg\|\sup_{\varepsilon>0} \int_{2B} h(x)T_{ - \varepsilon x}R_{j}\mathbf{x}\,dx \bigg\|_{L_{p}(\mathbb{R}^{d}_{\theta})}\lesssim \|R_{j}\mathbf{x}\|_{L_{p}(\mathbb{R}^{d}_{\theta})},
	\end{align}
	which gives an upper bound for the  first term of (\ref{24731.1}).
	
	Now, we pay attention to the summation term of (\ref{24731.1}). For any $k\geq 1$, applying  (\ref{0118.1}) again, we have 
	\begin{align*}
		\bigg\|\sup_{\varepsilon>0} \int_{2^{k}<|x|\leq 2^{k+1}} &h(x)T_{ -  \varepsilon x}R_{j}\mathbf{x}\,dx \bigg\|_{L_{p}(\mathbb{R}^{d}_{\theta})}\\
		&=\sup_{\tilde{J}\subset\mathbb{R}_{+},\,\tilde{J} \mbox{ finite}}	\bigg\|\sup_{\varepsilon\in \tilde{J}} \int_{2^{k}<|x|\leq 2^{k+1}} h(x)T_{ -\varepsilon x}R_{j}\mathbf{x}\,dx \bigg\|_{L_{p}(\mathbb{R}^{d}_{\theta})}.
	\end{align*}
	In the following, we fix  $\tilde{J}$ and choose a compact   subset $\widetilde{K}\subset \mathbb{R}^d$ such that $2^{k+1}\varepsilon B$ is contained in $\widetilde{K}$ for all $\varepsilon\in \tilde{J}$. For any ${\delta}>0$, Lemma \ref{240610.2} implies that there exists a compact subset $\widetilde{F}\subset\mathbb{R}^d$ with $|\widetilde{F}|>0$ satisfying  $\frac{|\widetilde{K}-\widetilde{F}|}{|\widetilde{F}|}\leq 1+{\delta}$. Since for any $t>0$, $T_{t}$ extends to an isometry on $L_{p}( \mathbb{R}^{d}_{\theta};l_{\infty}(\mathbb{R}_{+}))$, we have  
	\begin{align*}
		\bigg\|\sup_{\varepsilon\in \tilde{J}} \int_{2^{k}<|x|\leq 2^{k+1}} &h(x)T_{ - \varepsilon x}R_{j}\mathbf{x}\,dx \bigg\|^{p}_{L_{p}(\mathbb{R}^{d}_{\theta})}\\
		= &\frac{1}{|\widetilde{F}|}\int_{\widetilde{F}}\bigg\|\sup_{\varepsilon\in \tilde{J}} T_{-t}\int_{2^{k}<|x|\leq 2^{k+1}} h(x)T_{ t-  \varepsilon x}R_{j}\mathbf{x}\,dx \bigg\|^{p}_{L_{p}(\mathbb{R}^{d}_{\theta})}dt\\
		=& \frac{1}{|\widetilde{F}|}\int_{\widetilde{F}}\bigg\|\sup_{\varepsilon\in \tilde{J}} \int_{2^{k}<|x|\leq 2^{k+1}} h(x)T_{ t-  \varepsilon x}R_{j}\mathbf{x}\,dx \bigg\|^{p}_{L_{p}(\mathbb{R}^{d}_{\theta})}dt.
	\end{align*}
	Decomposing   $T_{t - \varepsilon x}R_{j}\mathbf{x}$ into the linear combination of four positive parts and using the triangle inequality,  
	above   equality is controlled by the sum of the following four terms:
	\begin{align*}
		\uppercase\expandafter{\romannumeral2}_{k,1}= \frac{1}{|\widetilde{F}|}\int_{\widetilde{F}}\bigg\|\sup_{\varepsilon\in \tilde{J}} \int_{2^{k}<|x|\leq 2^{k+1}} h(x)\Re^{+}(
		T_{t-\varepsilon x}R_{j}\mathbf{x}) dx \bigg\|_{L_{p}(\mathbb{R}^{d}_{\theta})}^{p}dt,\\
		\uppercase\expandafter{\romannumeral2}_{k,2}=\frac{1}{|\widetilde{F}|}\int_{\widetilde{F}}\bigg\|\sup_{\varepsilon\in \tilde{J}} \int_{2^{k}<|x|\leq 2^{k+1}} h(x)\Re^{-}(
		T_{t-\varepsilon x}R_{j}\mathbf{x}) dx \bigg\|_{L_{p}(\mathbb{R}^{d}_{\theta})}^{p}dt,\\
		\uppercase\expandafter{\romannumeral2}_{k,3}=\frac{1}{|\widetilde{F}|}\int_{\widetilde{F}}\bigg\|\sup_{\varepsilon\in \tilde{J}} \int_{2^{k}<|x|\leq 2^{k+1}} h(x)\Im^{+}(	T_{t- \varepsilon x}R_{j}\mathbf{x}) dx \bigg\|_{L_{p}(\mathbb{R}^{d}_{\theta})}^{p}dt,\\
		\uppercase\expandafter{\romannumeral2}_{k,4}=\frac{1}{|\widetilde{F}|}\int_{\widetilde{F}}\bigg\|\sup_{\varepsilon\in \tilde{J}} \int_{2^{k}<|x|\leq 2^{k+1}} h(x)\Im^{-}(	T_{t- \varepsilon x}R_{j}\mathbf{x}) dx \bigg\|_{L_{p}(\mathbb{R}^{d}_{\theta})}^{p}dt.
	\end{align*}
	
	We first estimate $\uppercase\expandafter{\romannumeral2}_{k,1}$.  The fact that
	$|h(x)|\leq\frac{C}{|x|^{d+1}}$ for $|x|\geq 2$ shows that
	\begin{align*}
		\int_{2^{k}<|x|\leq 2^{k+1}} h(x)\Re^{+}(
		T_{t-\varepsilon x}R_{j}\mathbf{x}) dx\leq &\int_{2^{k}<|x|\leq 2^{k+1}} |h(x)|\Re^{+}(
		T_{t-\varepsilon x}R_{j}\mathbf{x}) dx\\
		\leq&C \int_{2^{k}<|x|\leq 2^{k+1}}\frac{1}{2^{k(d+1)}}\Re^{+}(
		T_{t-  \varepsilon x}R_{j}\mathbf{x}) dx\\
		\leq&C_{d}\,2^{-k}  \frac{1}{|2^{k+1}B|}\int_{2^{k+1}B}\Re^{+}(
		T_{t-  \varepsilon x}R_{j}\mathbf{x}) dx\\
		=&C_{d}\,2^{-k}  \frac{1}{|2^{k+1} \varepsilon B|}\int_{t-2^{k+1} \varepsilon B}\Re^{+}(
		T_{x}R_{j}\mathbf{x}) dx.	 
	\end{align*}
	Since $t\in \tilde{F}$ and $\cup_{\varepsilon\in \tilde{J}}2^{k+1}\varepsilon B\subset \tilde{K}$, we have
	\begin{align*}
		\int_{2^{k}<|x|\leq 2^{k+1}} h(x)\Re^{+}(
		T_{t-\varepsilon x}R_{j}\mathbf{x}) dx\leq C_{d}\,2^{-k}  \frac{1}{|2^{k+1} \varepsilon B|}\int_{t-2^{k+1} \varepsilon B}\Re^{+}(
		T_{x}R_{j}\mathbf{x}\,\chi_{\tilde{F} -\tilde{K} }(x)) dx.
	\end{align*}
	Applying Proposition \ref{241227.1} to the function $ \Re^{+}(
	T_{(\cdot)}R_{j}\mathbf{x}\,\chi_{\tilde{F} -\tilde{K}}(\cdot)) \in L_{p}(\mathbb{R}^d;L_{p}(\mathbb{R}^d_{\theta}))$, there exists a positive operator-valued function $F_{k,1}\in L_{p}(\mathbb{R}^d;L_{p}(\mathbb{R}^d_{\theta}))$ such that  
	\begin{align*} 
		\frac{1}{|2^{k+1}\varepsilon B|}\int_{t-2^{k+1}\varepsilon B}\Re^{+}(
		T_{x}R_{j}\mathbf{x}\,\chi_{\tilde{F}-\tilde{K}}(x))dx\leq F_{k,1}(t),~~\forall \varepsilon\in \tilde{J},
	\end{align*}
	with
	\begin{align}\label{241215.7} 
		\|F_{k,1}\|_{L_{p}(\mathbb{R}^d;L_{p}(\mathbb{R}^d_{\theta}))}\lesssim \|\Re^{+}(
		T_{(\cdot)}R_{j}\mathbf{x}\,\chi_{\tilde{K}-\tilde{F}}(\cdot))\|_{L_{p}(\mathbb{R}^d;L_{p}(\mathbb{R}^d_{\theta}))}.
	\end{align} 
	We conclude that
	\begin{align}\label{241213.2}
		-2^{-k}F_{k,1}(t)\lesssim\int_{2^{k}<|x|\leq 2^{k+1}} h(x)\Re^{+}(
		T_{t-\varepsilon x}R_{j}\mathbf{x}) dx\lesssim  2^{-k}F_{k,1}(t),\quad \forall\varepsilon\in \tilde{J}.
	\end{align} 
	Combining   (\ref{24829.2}), (\ref{241215.7}) and (\ref{241213.2}),   we have
	\begin{align*}
		\uppercase\expandafter{\romannumeral2}_{k,1}\lesssim& 2^{-kp}\frac{1}{|\widetilde{F}|}	\|F_{k,1}\|^{p}_{L_{p}(\mathbb{R}^d;L_{p}(\mathbb{R}^d_{\theta}))}\\
		\lesssim& 2^{-kp}\frac{1}{|\widetilde{F}|}\|
		T_{(\cdot)}R_{j}\mathbf{x}\,\chi_{\widetilde{K}-\widetilde{F}}(\cdot)\|^p_{L_{p}(\mathbb{R}^d;L_{p}(\mathbb{R}^d_{\theta}))}\\
		\lesssim& 2^{-kp}\frac{|\widetilde{F}-\widetilde{K}|}{|\widetilde{F}|}\|R_{j}\mathbf{x}\|^p_{L_{p}(\mathbb{R}^{d}_{\theta})}\\
		\lesssim& 2^{-kp}(1+\delta)\|R_{j}\mathbf{x}\|^p_{L_{p}(\mathbb{R}^{d}_{\theta})}.
	\end{align*}
	The same method works for other three terms $\uppercase\expandafter{\romannumeral2}_{k,2}$, $\uppercase\expandafter{\romannumeral2}_{k,3}$,  and $\uppercase\expandafter{\romannumeral2}_{k,4}$ as well. Thus, we obtain  
	\begin{align}\label{241014.6}
		\bigg\|\sup_{\varepsilon\in \tilde{J}} \int_{2^{k}<|x|\leq 2^{k+1}} h(x)T_{ - \varepsilon x}R_{j}\mathbf{x}\, dx \bigg\|^{p}_{L_{p}(\mathbb{R}^{d}_{\theta})}\lesssim 2^{-kp}(1+\delta)\|R_{j}\mathbf{x}\|^p_{L_{p}(\mathbb{R}^{d}_{\theta})}.
	\end{align}
	The  implicit constant in  (\ref{241014.6})  is independent of $\tilde{J}$. Since $\tilde{J}$ and $\delta$ are all arbitrarily chosen, we conclude that
	\begin{align}\label{24424.2}
		\bigg\|\sup_{\varepsilon>0} \int_{2^{k}<|x|\leq 2^{k+1}} h(x)T_{ - \varepsilon x}R_{j}\mathbf{x}\, dx \bigg\|_{L_{p}(\mathbb{R}^{d}_{\theta})}\lesssim 2^{-k}\|R_{j}\mathbf{x}\|_{L_{p}(\mathbb{R}^{d}_{\theta})}.
	\end{align}
	
	Finally, combining (\ref{24731.1}), (\ref{24424.1}), and (\ref{24424.2}), we get 
	\begin{align*}
		\|\sup_{\varepsilon>0}R_{j}^{\varepsilon}\mathbf{x}\|_{L_{p}(\mathbb{R}^{d}_{\theta})}\lesssim \bigg(1+\sum_{k\geq 1}2^{-k}\bigg)\|R_{j}\mathbf{x}\|_{L_{p}(\mathbb{R}^{d}_{\theta})}\lesssim \|R_{j}\mathbf{x}\|_{L_{p}(\mathbb{R}^{d}_{\theta})},
	\end{align*}
	which completes the proof. $\hfill\square$
	
	\vspace{0.2cm}
	
	Similar to the quantum tori, the constant $C_{d,p}$ in Theorem \ref{24521.3} is independent of the dimension  when $2\leq p<\infty$. To get it, we need  an analogue of Proposition \ref{24325.3} in quantum Euclidean space.
	\begin{proposition}\label{24725.8}
		Let $\psi\in L_{1}(\mathbb{R}^{d})$ and $\mathbf{x}\in L_{p}(\mathbb{R}^d_{\theta})$ with $1<p<\infty$. Then, we have
		\begin{align*}
			\bigg\|\sup_{r>0}\frac{1}{r}\int_{0}^{r}\psi_{t}*_{\theta}\mathbf{x}\,dt\bigg\|_{L_{p}(\mathbb{R}^{d}_{\theta})}\leq 4C_{p} \|\psi\|_{L_{1}(\mathbb{R}^d)}\|\mathbf{x}\|_{L_{p}(\mathbb{R}^d_{\theta})},
		\end{align*}
		where $C_{p}$ is the $L_{p}$ norm of the one-dimensional  semi-commutative Hardy-Littlewood operator.
	\end{proposition}
	\begin{proof}
		We first  assume that $\mathbf{x}$ is positive. 	Using   the noncommutative   convolution $*_{\theta}$  defined in (\ref{24726.3}), it is straightforward to see
		\begin{align*}
			\frac{1}{r}\int_{0}^{r}\psi_{t}*_{\theta}\mathbf{x}\,dt=&\frac{1}{r}\int_{0}^{r}\int_{\mathbb{R}^d} \psi_{t}(y)T_{-y}\mathbf{x}\,dydt\\
			=&\int_{\mathbb{R}^d} \psi(y)\frac{1}{r}\int_{0}^{r}T_{-ty}\mathbf{x}\,dtdy.
		\end{align*}
		Once we prove
		\begin{align}\label{24724.2}
			\bigg\|\sup_{r>0}\frac{1}{r}\int_{0}^{r}T_{-t y}\mathbf{x}\,dt\bigg\|_{L_{p}(\mathbb{R}^{d}_{\theta})}\leq  C_{p}\|\mathbf{x}\|_{L_{p}(\mathbb{R}^{d}_{\theta})},
		\end{align}
		the following estimate  will follow from  (\ref{24829.2}),
		\begin{align*}
			\bigg\|\sup_{r>0}\frac{1}{r}\int_{0}^{r}\psi_{t}*_{\theta}\mathbf{x}\,dt\bigg\|_{L_{p}(\mathbb{R}^{d}_{\theta})}=&\bigg\|\sup_{r>0}\int_{\mathbb{R}^d} \psi(y)\frac{1}{r}\int_{0}^{r}T_{-ty}\mathbf{x}\,dtdy\bigg\|_{L_{p}(\mathbb{R}^{d}_{\theta})}\\
			\leq& \int_{\mathbb{R}^d}|\psi(y)|\bigg\| \sup_{r>0}\frac{1}{r}\int_{0}^{r}T_{-t y}\mathbf{x}\,dt\bigg\|_{L_{p}(\mathbb{R}^{d}_{\theta})}dy\\
			\leq&  C_{p}\|\psi\|_{L_{1}(\mathbb{R}^d)}\|\mathbf{x}\|_{L_{p}(\mathbb{R}^{d}_{\theta})}.
		\end{align*}
		For general $\mathbf{x} \in L_{p}(\mathbb{R}^d_{\theta})$, we decompose it into the linear combination of four positive parts, then this proposition is proved.

		Now, we return to proving (\ref{24724.2}). From (\ref{0118.1}), it follows that 
		\begin{align*}
			\bigg\|\sup_{r>0}\frac{1}{r}\int_{0}^{r}T_{-t y} \mathbf{x}\,dt\bigg\|_{L_{p}(\mathbb{R}^{d}_{\theta})}=\sup_{J\subset\mathbb{R}_{+},\,J \mbox{ finite}}\bigg\|\sup_{r\in J}\frac{1}{r}\int_{0}^{r}T_{-ty}\mathbf{x}\,dt\bigg\|_{L_{p}(\mathbb{R}^{d}_{\theta})}.
		\end{align*}
		In the following, we fix $J$ and choose a compact subset ${K}\subset \mathbb{R}^d$ such that $\left\{ -ty:\,\, 0<t<r\right\}$ is contained in ${K}$ for all $r\in J$.  Lemma \ref{240610.2}  shows that,  for any ${\delta}>0$, there exists a compact subset $F \subset \mathbb{R}^d$ with $|F|>0$ such that  $\frac{|{K}-{F}|}{|{F}|}\leq 1+{\delta}$. According to Remark 
		\ref{241215.5} and \cite[Proposition 5.3]{MR4611833},  for any $t>0$, $T_{t}$ extends to an isometry on $L_{p}( \mathbb{R}^{d}_{\theta};l_{\infty}(\mathbb{R}_{+}))$. Hence,
		\begin{align*}
			\bigg\|\sup_{r\in J}\frac{1}{r}\int_{0}^{r}T_{-t y} \mathbf{x}\,dt\bigg\|^p_{L_{p}(\mathbb{R}^{d}_{\theta})}=&\frac{1}{|F|}\int_{F}	\bigg\|\sup_{r\in J}\frac{1}{r}\int_{0}^{r}T_{-t y} \mathbf{x}\,dt\bigg\|^p_{L_{p}(\mathbb{R}^{d}_{\theta})}dz\\
			=&\frac{1}{|F|}\int_{F}	\bigg\|\sup_{r\in J}\frac{1}{r}\int_{0}^{r}T_{z-t y} \mathbf{x}\,dt\bigg\|^p_{L_{p}(\mathbb{R}^{d}_{\theta})}dz\\
			\leq &\frac{1}{|F|}\bigg\|\sup_{r\in J}\frac{1}{r}\int_{0}^{r}T_{z-t y} \mathbf{x}\,\chi_{F-K}(z-ty)dt\bigg\|^p_{L_{p}(\mathbb{R}^{d};L_{p}(\mathbb{R}^{d}_{\theta}))}.
		\end{align*}
		Due to (\ref{24065.7}), the above estimate is bounded by 
		\begin{align*}
			\frac{C_{p}^{p}}{|F|} \|T_{(\cdot)}\mathbf{x}\, \chi_{F-K}(\cdot) \|^p_{L_{p}(\mathbb{R}^{d};L_{p}(\mathbb{R}^{d}_{\theta}))}\leq   C_{p}^{p}(1+\delta) \|\mathbf{x}\|^p_{L_{p}(\mathbb{R}^{d}_{\theta})}.
		\end{align*}
		Finally, the arbitrariness of $J$ and $\delta$ implies that (\ref{24724.2}) holds. 
	\end{proof}
	
	\begin{remark}\label{241226.4}
		Combining  Remark \ref{240620.1} and  Proposition \ref{24725.8}, if $\psi\in L_{1}(\mathbb{R}^{d})$ and $\mathbf{x}\in L_{2}(\mathbb{R}^d_{\theta})$,  we obtain
		\begin{align*}
			\bigg\|\sup_{r>0}\frac{1}{r}\int_{0}^{r}\psi_{t}*_{\theta}\mathbf{x}\,dt\bigg\|_{L_{2}(\mathbb{R}^{d}_{\theta})}\leq 96 \|\psi\|_{L_{1}(\mathbb{R}^d)}\|\mathbf{x}\|_{L_{2}(\mathbb{R}^d_{\theta})}.
		\end{align*}
	\end{remark}

	\vspace{0.1cm}
	
	\noindent\textbf{Proof of Theorem \ref{24728.1}}: Let $\mathbf{x}=U(f)\in \mathcal{S}(\mathbb{R}^{d}_{\theta})$ and $A^{t}$ is defined by the following Fourier multiplier form:
	\begin{align*}
		A^{t}\mathbf{x}=\int_{\mathbb{R}^d}m(t|\xi|)f(\xi)U(\xi)d\xi.
	\end{align*}
	It was shown in \cite[Remark 1]{liu2023p} that $m(t|\xi|)=\widehat{\phi_{t}}(\xi)$, where $\phi_{t}$ is the function given in (\ref{24919.1}). Therefore, we write   $A^{t}\mathbf{x}$   as
	\begin{align*}
		A^{t}\mathbf{x}=&\int_{\mathbb{R}^d}f(\xi)U(\xi)\widehat{\phi_{t}}(\xi)d\xi\\
		=&\int_{\mathbb{R}^d}\phi_{t}(y)\int_{\mathbb{R}^d}f(\xi)e^{-2\pi \mathrm{i} \langle y,\xi  \rangle}U(\xi)d\xi dy\\
		=&\int_{\mathbb{R}^d}\phi_{t}(y)T_{-  y}\mathbf{x}\,dy.
	\end{align*}
	
	We  first prove that
	\begin{align}\label{24726.2}
		\|\sup_{t>0}A^{t}\mathbf{x}\|_{L_{2}(\mathbb{R}^{d}_{\theta})}\leq (96+2\sqrt{2}) \|\mathbf{x}\|_{L_{2}(\mathbb{R}^{d}_{\theta})}.
	\end{align}
	Using integration by parts, we have 
	\begin{align}\label{24725.1}
		\phi_{t}(y)=\frac{1}{t}\int_{0}^{t}\phi_{s}(y)ds+\frac{1}{t}\int_{0}^{t}s\frac{d}{ds}\phi_{s}(y)ds.
	\end{align}
	Combining the triangle inequality with (\ref{24725.1}), the following estimate holds, 
	\begin{align}\label{24724.1}
		\|\sup_{t>0}A^{t}\mathbf{x}\|_{L_{2}(\mathbb{R}^{d}_{\theta})}\leq& \bigg\|\sup_{t>0}\frac{1}{t}\int_{0}^{t}\int_{\mathbb{R}^d}\phi_{s}(y)T_{-  y}\mathbf{x}\, dyds\bigg\|_{L_{2}(\mathbb{R}^{d}_{\theta})}\nonumber\\
		+&\bigg\|\sup_{t>0}\frac{1}{t}\int_{0}^{t}\int_{\mathbb{R}^d}s\frac{d}{ds}\phi_{s}(y)T_{-  y}\mathbf{x}\, dyds\bigg\|_{L_{2}(\mathbb{R}^{d}_{\theta})}.
	\end{align}
	
	By Lemma \ref{241212.4} and Remark \ref{241226.4},  the first term on the  right-hand side of (\ref{24724.1}) is bounded by $96\|\mathbf{x}\|_{L_{2}(\mathbb{R}^{d}_{\theta})}$, i.e.,
	\begin{align}\label{24726.1}
		\bigg\|\sup_{t>0}\frac{1}{t}\int_{0}^{t}\int_{\mathbb{R}^d}\phi_{s}(y)T_{- y}\mathbf{x}\, dyds\bigg\|_{L_{2}(\mathbb{R}^{d}_{\theta})}=\bigg\|\sup_{t>0}\frac{1}{t}\int_{0}^{t} \phi_{s}*_{\theta} \mathbf{x}\,ds\bigg\|_{L_{2}(\mathbb{R}^{d}_{\theta})}
		\leq96\|\mathbf{x}\|_{L_{2}(\mathbb{R}^{d}_{\theta})}.
	\end{align}

	In the following, we  decompose $\int_{\mathbb{R}^d}s\frac{d}{ds}\phi_{s}(y)T_{-  y}\mathbf{x}\,dy$ into the linear combination of four positive parts. Then, according to the triangle inequality, the second term on the  right-hand side of (\ref{24724.1}) is controlled by the sum of the following  four terms:
	\begin{align*}
		\uppercase\expandafter{\romannumeral3}_{1}=\bigg\|\sup_{t>0}\frac{1}{t}\int_{0}^{t}\Re_{+}\bigg(\int_{\mathbb{R}^d}s\frac{d}{ds}\phi_{s}(y)T_{-  y}\mathbf{x}\, dy\bigg)ds\bigg\|_{L_{2}(\mathbb{R}^{d}_{\theta})},\\
		\uppercase\expandafter{\romannumeral3}_{2}=\bigg\|\sup_{t>0}\frac{1}{t}\int_{0}^{t}\Re_{-}\bigg(\int_{\mathbb{R}^d}s\frac{d}{ds}\phi_{s}(y)T_{-  y}\mathbf{x}\, dy\bigg)ds\bigg\|_{L_{2}(\mathbb{R}^{d}_{\theta})},\\
		\uppercase\expandafter{\romannumeral3}_{3}=\bigg\|\sup_{t>0}\frac{1}{t}\int_{0}^{t}\Im_{+}\bigg(\int_{\mathbb{R}^d}s\frac{d}{ds}\phi_{s}(y)T_{-  y}\mathbf{x}\, dy\bigg)ds\bigg\|_{L_{2}(\mathbb{R}^{d}_{\theta})},\\
		\uppercase\expandafter{\romannumeral3}_{4}=\bigg\|\sup_{t>0}\frac{1}{t}\int_{0}^{t}\Im_{-}\bigg(\int_{\mathbb{R}^d}s\frac{d}{ds}\phi_{s}(y)T_{-  y}\mathbf{x}\, dy\bigg)ds\bigg\|_{L_{2}(\mathbb{R}^{d}_{\theta})}.
	\end{align*}
	We first deal with $\uppercase\expandafter{\romannumeral3}_{1}$. The H\"{o}lder inequality (\ref{3802}) shows  that
	\begin{align}\label{24725.4}
		\frac{1}{t}\int_{0}^{t}\Re_{+}\bigg(\int_{\mathbb{R}^d}s\frac{d}{ds}\phi_{s}(y)T_{- y}\mathbf{x}\,dy\bigg)ds\leq \bigg( \int_{0}^{\infty}\bigg(\Re_{+}\bigg(\int_{\mathbb{R}^d}s\frac{d}{ds}\phi_{s}(y)T_{- y}\mathbf{x}\,dy\bigg)\bigg)^2\frac{ds }{s}  \bigg)^{\frac{1}{2}}.
	\end{align}
	Then we apply  (\ref{24829.2}) and (\ref{24725.4})   to get
	\begin{align*}
		\uppercase\expandafter{\romannumeral3}_{1}\leq&\bigg\|\bigg( \int_{0}^{\infty}\bigg(\Re_{+}\bigg(\int_{\mathbb{R}^d}s\frac{d}{ds}\phi_{s}(y)T_{- y}\mathbf{x}\, dy\bigg)\bigg)^2\frac{ds }{s}  \bigg)^{\frac{1}{2}}\bigg\|_{L_{2}(\mathbb{R}^{d}_{\theta})}\\
		\leq &\bigg(\int_{0}^{\infty}\bigg\| \int_{\mathbb{R}^d}s\frac{d}{ds}\phi_{s}(y)T_{- y}\mathbf{x}\, dy\bigg\|^2_{L_{2}(\mathbb{R}^{d}_{\theta})}\frac{ds }{s}  \bigg)^{\frac{1}{2}}\\
		=&\bigg(\int_{0}^{\infty}\bigg\| s\frac{d}{ds}\phi_{s}*_{\theta}\mathbf{x}\bigg\|^2_{L_{2}(\mathbb{R}^{d}_{\theta})}\frac{ds }{s}  \bigg)^{\frac{1}{2}}.
	\end{align*}
	Using Plancherel identity (\ref{241225.2}),   (\ref{24726.4}), and Lemma \ref{241212.4},  we have
	\begin{align*}
		\uppercase\expandafter{\romannumeral3}_{1}\leq  	 \bigg(\int_{0}^{\infty}\bigg\|s\frac{s }{ ds}\widehat{\phi_{s}} \, f\bigg\|^2_{L_{2}(\mathbb{R}^{d})}\frac{ds  }{ s}\bigg)^{\frac{1}{2}}
		\leq \frac{\sqrt{2}}{2}\|\mathbf{x} \|_{L_{2}(\mathbb{R}^{d}_{\theta})}.
	\end{align*}
	The other three terms   $\uppercase\expandafter{\romannumeral3}_{2}$,  $\uppercase\expandafter{\romannumeral3}_{3}$, and $\uppercase\expandafter{\romannumeral3}_{4}$  are handled similarly, which leads to
	\begin{align}\label{24725.6}
		\bigg\|\sup_{t>0}\frac{1}{t}\int_{0}^{t}\int_{\mathbb{R}^d}s\frac{d}{ds}\phi_{s}(y)T_{- y}\mathbf{x}\, dyds\bigg\|_{L_{2}(\mathbb{R}^{d}_{\theta})}\leq 2\sqrt{2} \|\mathbf{x} \|_{L_{2}(\mathbb{R}^{d}_{\theta})}.
	\end{align}
	Combining the preceding estimates (\ref{24724.1}), (\ref{24726.1}), and (\ref{24725.6}), we arrive at   (\ref{24726.2}).
	
	As for $p=\infty$, we use Lemma \ref{241212.4} to deduce 
	\begin{align}\label{24725.5}
		\|\sup_{t>0}A^{t}\mathbf{x} \|_{L_{\infty}(\mathbb{R}^{d}_{\theta})}
		=  \sup_{t>0}\bigg\|\int_{\mathbb{R}^d}\phi_{t}(y)T_{- y}\mathbf{x}\, dy\bigg\|_{L_{\infty}(\mathbb{R}^{d}_{\theta})}
		\leq \|\mathbf{x} \|_{L_{\infty}(\mathbb{R}^{d}_{\theta})}.
	\end{align}
	Finally, taking an interpolation with    (\ref{24726.2}) and (\ref{24725.5}),    we   get for any  $\mathbf{x}\in L_{p}(\mathbb{R}^{d}_{\theta})$,
	\begin{align}\label{241016.1}
		\|\sup_{t>0}A^{t}\mathbf{x} \|_{L_{p}(\mathbb{R}^{d}_{\theta})}\leq (96+2\sqrt{2})^{\frac{2}{p}} \|\mathbf{x} \|_{L_{p}(\mathbb{R}^{d}_{\theta})}, ~~~2\leq p<\infty.
	\end{align}
	
	Now, we use (\ref{241016.1}) to get our main results in Theorem \ref{24728.1}. For any $t>0$, the  truncated Riesz transform  factorization  $R_{j}^{t}\mathbf{x}=A^{t}(R_{j}\mathbf{x})$ holds for $\mathcal{S}(\mathbb{R}^{d}_{\theta})$, which implies    the optimal constant $C_{d,p}$ in the inequality
	\begin{align*}
		\|\sup_{t>0}R_{j}^{t}\mathbf{x}\|_{L_{p}(\mathbb{R}^{d}_{\theta})}\leq C_{d,p}\|R_{j}\mathbf{x}\|_{L_{p}(\mathbb{R}^{d}_{\theta})}
	\end{align*}
	is controlled by
	the optimal constant $D_{d,p}$ in the inequality
	\begin{align*}
		\|\sup_{t>0}A^{t}\mathbf{x}\|_{L_{p}(\mathbb{R}^{d}_{\theta})}\leq D_{d,p}\|\mathbf{x}\|_{L_{p}(\mathbb{R}^{d}_{\theta})}.
	\end{align*}
	Consequently, we obtain
	\begin{align*}
		\|\sup_{\varepsilon>0}R_{j}^{\varepsilon}\mathbf{x}\|_{L_{p}(\mathbb{R}^{d}_{\theta})}\leq (96+2\sqrt{2})^{\frac{2}{p}}\|R_{j}\mathbf{x}\|_{L_{p}(\mathbb{R}^{d}_{\theta})}.
	\end{align*} 
	$\hfill\square$
	
	\vspace{0.5cm}
	
	\noindent\textbf{\small Acknowledgements} {\small The authors would like to express their gratitude to the referee for his comments that improve the exposition of the paper.  Lai, X. is supported by National Natural Science Foundation of China (No. 12271124, No. 12322107 and No. W2441002) and Heilongjiang Provincial Natural Science Foundation of China (YQ2022A005).  Xiong, X. is supported by the National Natural Science Foundation of China (No. 12371138).}
	
	\vspace{0.5cm}
	
	\noindent\textbf{\large Declarations}
	
	\vspace{0.2cm}
	
	\noindent\textbf{Conflict of interest} The authors confirm that there are no conflicts of interest or funding related to this work.

\end{document}